\documentclass[leqno,final]{article}

\usepackage[dvips]{graphicx}
\usepackage{amsmath,amstext,amssymb,amsthm}

\numberwithin{equation}{section}
\newtheorem{mythm}{Theorem}[section]
\newtheorem{mylem}{Lemma}[section]
\newtheorem{mycor}{Corollary}[section]
\newtheorem{myrem}{Remark}[section]

\newcommand{\norm}[1]{\left\Vert#1\right\Vert}
\newcommand{\norml}[2]{\left\Vert#1\right\Vert_{L^2(#2)}}
\newcommand{\norme}[1]{\left\Vert{\hskip -2.7pt}\left\vert #1 \right\vert{\hskip -2.7pt}\right\Vert}
\newcommand{\abs}[1]{\left\vert#1\right\vert}
\newcommand{\pd}[1]{\left\langle #1\right\rangle}
\newcommand{\pdjj}[2]{\left\langle \left[#1\right], \left[#2\right]\right\rangle_e}

\newcommand{\set}[1]{\left\{#1\right\}}

\newcommand{\jm}[1]{\left[#1\right]}

\newcommand{\R}{\mathbb{R}}

\newcommand{\csta}{C_{\rm sta}}
\newcommand{\cerr}{C_{\rm err}}

\newcommand{\db}{\displaybreak[0]}
\newcommand{\nn}{\nonumber}
\newcommand{\ls}{\lesssim}

\newcommand{\De}{\Delta}
\newcommand{\ep}{\varepsilon}
\newcommand{\ga}{\gamma}
\newcommand{\Ga}{\Gamma}

\newcommand{\na}{\nabla}
\newcommand{\om}{\omega}
\newcommand{\Om}{\Omega}
\newcommand{\pa}{\partial}
\newcommand{\pr}{\prime}
\newcommand{\si}{\sigma}

\newcommand{\ze}{\zeta}

\renewcommand{\i}{{\rm\mathbf i}}
\newcommand{\ue}{u_\mathcal{E}}
\newcommand{\ua}{u_\mathcal{A}}

\DeclareMathOperator{\re}{{Re}} 
\DeclareMathOperator{\im}{{Im}}
\DeclareMathOperator{\ios}{{\mathcal{I}_{\rm Os}}}

\newcommand{\p}{\partial}
\newcommand{\T}{\mathcal{T}}
\newcommand{\E}{\mathcal{E}}

\title{Pre-asymptotic error Analysis of CIP-FEM and FEM\\ for Helmholtz Equation with high
Wave Number. \\Part II: $hp$ version}
\author{
Lingxue Zhu
\thanks{Department of Mathematics, Nanjing University, Jiangsu,
210093, P.R. China. ({\tt zhulingxue@163.com}). }
\and 
Haijun Wu
\thanks{Department of Mathematics, Nanjing University, Jiangsu,
210093, P.R. China. ({\tt hjw@nju.edu.cn}). The work of the second author was
partially supported by the National Magnetic Confinement Fusion Science Program under grant 2011GB105003 and by the NSF of China grants 10971096, 11071116, 91130004.}
}

\begin{document}
\date{}
\maketitle


\setcounter{page}{1}

\begin{abstract}
In this paper, which is part II in a series of two, the pre-asymptotic error analysis of the continuous interior penalty finite element method (CIP-FEM) and the FEM for the Helmholtz equation in two and three dimensions is continued. While part I contained results on the linear CIP-FEM and FEM, the present part deals with approximation spaces of order $p \ge 1$. By using a modified duality argument, pre-asymptotic error estimates are derived for both methods under the condition of $\frac{kh}{p}\le C_0\big(\frac{p}{k}\big)^{\frac{1}{p+1}}$, where $k$ is the wave number, $h$ is the mesh size, and $C_0$ is a constant independent of $k, h, p$, and the penalty parameters. It is shown that the pollution errors of both methods in $H^1$-norm are $O(k^{2p+1}h^{2p})$ if $p=O(1)$ and are $O\Big(\frac{k}{p^2}\big(\frac{kh}{\sigma p}\big)^{2p}\Big)$ if the exact solution $u\in H^2(\Om)$ which coincide with existent dispersion analyses for the FEM on Cartesian grids. Here $\si$ is a constant independent of $k, h, p$, and the penalty parameters. Moreover, it is proved that the CIP-FEM is stable for any $k, h, p>0$ and penalty parameters with positive imaginary parts. Besides the advantage of the absolute stability of the CIP-FEM compared to the FEM, the penalty parameters may be tuned to reduce the pollution effects.
\end{abstract}




\section{Introduction}\label{sec-1} This is the second installment in a series (cf. \cite{w}) which is devoted to pre-asymptotic stability and error estimates of $hp$ versions of some continuous interior penalty finite element method (CIP-FEM) and the finite element method (FEM) for the following Helmholtz problem:
\begin{alignat}{2}
-\De u - k^2 u &=f  &&\qquad\mbox{in  } \Om,\label{eq1.1a}\\
\frac{\pa u}{\pa n} +\i k u &=g &&\qquad\mbox{on } \Ga,\label{eq1.1b}
\end{alignat}
where $\Om\subset \R^d,\, d=2,3$ is a domain with smooth boundary,
$\Ga:=\pa\Om$, $\i=\sqrt{-1}$ denotes the imaginary unit, and $n$
denotes the unit outward normal
to $\pa\Om$. The above Helmholtz problem is an approximation of the following
acoustic scattering problem (with time dependence $e^{\i\om t}$):
\begin{alignat}{2}\label{e1.1a}
-\De u- k^2 u &= f \qquad &&\mbox{in } \R^d,\\
\sqrt{r}\Bigl( \frac{\p (u-u^{\rm inc})}{\p r} + \i k (u-u^{\rm inc})\Bigr) &\rightarrow 0 &&\mbox{as }
r=|x|\rightarrow \infty, \label{e1.2a}
\end{alignat}
where  $u^{\rm inc}$ is the incident wave and $k$ is known as the wave number. The Robin boundary condition \eqref{eq1.1b} is known as the
first order approximation of the radiation condition \eqref{e1.2a} (cf. \cite{em79}).
We remark that the Helmholtz problem \eqref{eq1.1a}--\eqref{eq1.1b} also arises in applications
 as a consequence of frequency domain treatment of attenuated scalar waves 
(cf. \cite{dss94}). 

It is well-known that the error bound of the finite element solution to the Helmholtz problem \eqref{eq1.1a}--\eqref{eq1.1b} usually consists of two parts, one is the same order as the error of the best approximation of $u$ from the finite element space, another one is worse than the convergence order of the best approximation and dominates the error bound of the finite element solution for large wave number $k$ (relative to $h$ and $p$) \cite{Ainsworth04,bs00, bips95, dbb99,ey11b, fw09, fw11, harari97, ib95a,ib97, thompson06, w}. The second part is the so-called pollution error in the literature. We recall that, the term ``asymptotic error estimate" refers to the error estimate without pollution error and the term ``pre-asymptotic error estimate" refers to the estimate with non-negligible  pollution effect.  The latest asymptotic error analysis was given by Melenk and Sauter \cite{ms10,ms11}.  It is shown that, the $hp$-FEM is pollution-free in the $H^1$-norm if 
\begin{align}\label{emscond0}
\frac{k h}{p}+k\big(\frac{kh}{\sigma p}\big)^p \text{ is small enough},
\end{align}
where $h$ is the mesh size, $p$ is the polynomial degree of finite element space, and $\si$ is some positive constant independent of $k, h$, and $p$. In particular, the $hp$-FEM is pollution-free under either of the following two conditions:
\begin{equation}\label{emscond1a}
p=O(1) \text{ fixed independent of } k\text{ and } k^{1+\frac1p}h \text{ is small enough},
\end{equation} 
\begin{equation}\label{emscond2a}
p\ge c\ln k  \text{ and } \frac{k\,h}{p} \text{ is small enough},
\end{equation} 
where $c$ is some positive constant independent of $k, h$, and $p$.
 Although the above results improve greatly the previous results which require $\frac{k^2h}{p}$ small enough (cf. \cite{ib95a,ib97}), the pre-asymptotic error estimates are still worth to be studied based on but not limited to the following considerations:
\begin{itemize}
\item For fixed $p$, the pollution effect can be reduced substantially but cannot be avoided in principle (cf. \cite{bips95,bs00}).

\item  Implementations of high order FEMs, e.g. $p\ge \ln k$ for large $k$, are not easy for problems with complicated geometries or different materials.

\item Given a tolerance $\ep$, one wish to use as small as possible number of degrees of freedom, or in other words, as large as possible mesh size $h$,  to achieve this tolerance. We denote the maximum mesh size by $h(k,\ep)$. Then $h(k,\ep)$ always locates in the pre-asymptotic range for large $k$ because the pollution error dominates the error bound. For example, when $p=1$, for large $k$, the pollution term is $O(k^3h^2)$ (cf. \cite{ib95a, w}) and hence $h(k, \ep)=O\big((\ep/k^3)^{1/2}\big)$ which is not in the asymptotic range $(0, C/k^2)$ implied by the condition \eqref{emscond1a}.

\item The pre-asymptotic error analysis of FEM is hard for higher dimensional problems. To the best of the authors' knowledge, besides the first part of this series \cite{w} for the linear FEM, no work has been done ever since the pioneer works of Ihlenburg and Babu\v{s}ka \cite{ib95a,ib97} for one dimensional problems.
\end{itemize}

The purpose of this paper and the companion paper \cite{w} is twofold: One is to derive pre-asymptotic error estimates for the FEM for the Helmholtz equation in two and three dimensions, under the condition of $\frac{kh}{p}\le C_0\big(\frac{p}{k}\big)^{\frac{1}{p+1}}$ for some positive constant $C_0$ independent of $k, h,$ and $p$. Clearly, this condition extends the  asymptotic range given by \eqref{emscond0}. For example, when $p=O(1)$, the above condition is satisfied if $k^{1+\frac{1}{p+1}}h$ is small enough. This would
be the condition obtained for the $(p+1)$-th order FEM by ``standard" arguments (cf. \eqref{emscond1a}).  Another purpose is to analyze some CIP-FEM which is absolutely stable (that is, stable for any $k, h,$ and $p$) and capable of achieving much less pollution effect than the FEM for fixed $p$. The CIP-FEM, which was first proposed by Douglas and Dupont \cite{dd76} for elliptic and parabolic problems in 1970's and then successfully applied to convection-dominated problems as a stabilization technique \cite{burman05, be05,be07, bfh06, bh04}, modifies the sesquilinear form of the FEM by adding the following least squares term penalizing the jump of the gradient of the discrete solution at mesh interfaces:
\begin{align}\label{ejump}
J(u,v):=&\sum_{e\in\E_h^{I}}\i\ga \frac{h_e}{p^2} \pdjj{\frac{\pa u}{\pa n_e}}{\frac{\pa v}{\pa n_e}}.
\end{align}
Note that the penalty parameter $\i\ga$ is chosen as a complex number in this paper and \cite{w} instead of the real numbers as usual.  The paper \cite{w} has considered the linear case $p=1$ and this paper will be devoted to  the $hp$-version. To be precise, we obtain the following results:
\begin{itemize}
\item[(i)]  There exists
a constant $C_0>0$ independent of $k$, $h$, and $\ga$, 
such that if $k\gtrsim 1$,
$0\le\ga\lesssim 1$,  and
\end{itemize}
 \begin{equation}\label{econd1a}
\frac{kh}{p}\le C_0\Big(\frac{p}{k}\Big)^{\frac{1}{p+1}}, 
\end{equation}
\begin{itemize}
\item[] then the following pre-asymptotic error
estimates hold for both the CIP-FEM and the FEM:\
\begin{align*}
\norm{u-u_h}_{H^1(\Om)} & \le\left\{\begin{array}{l}
				C_1\Big(\dfrac{h}{p}+\dfrac{1}{p}\Big(\dfrac{kh}{\sigma p}\Big)^p\Big)+C_2\,\dfrac{k}{p^2}\Big(\dfrac{kh}{\sigma p}\Big)^{2p}, \quad\text{ if } u\in H^2(\Om),\\
				\\
				C_1 (kh)^p+C_2 k(k h)^{2p}, \quad\text{ if }p=O(1) \text{ and } \norm{u}_{H^{p+1}(\Om)}\ls k^p.
				\end{array}\right.
\end{align*}
Here $\si$ is a constant independent of $k, h, p$, and the penalty parameters. Note that the pollution term is $O(k^{2p+1}h^{2p})$ if $p=O(1)$ and $\norm{u}_{H^{p+1}(\Om)}\ls k^p$ and is $O\Big(\frac{k}{p^2}\big(\frac{kh}{\sigma p}\big)^{2p}\Big)$ if $u\in H^2(\Om)$.
\item[(ii)] The CIP-FEM attains a unique solution for any $k>0$, $h>0$, $p>0$, and $\ga>0$.
\item[(iii)] Estimates in the $L^2$-norm are also obtained.
\end{itemize}
We remark that the numerical tests in \cite{w}  for the linear CIP-FEM show that the penalty parameter may be tuned to greatly reduce the pollution errors.

Error analysis and dispersion analysis are two main tools to understand numerical behaviors in short wave computations. The later one, which is usually performed on \emph{structured} meshes, estimates the error between the wave number $k$ of the continuous problem and some discrete wave number (denoted by $\om$) of the numerical scheme \cite{Ainsworth04, Ainsworth09, dbb99, harari97, ib95a, ib97, thompson06, thompson94}. In particular, it is shown for the $hp$-FEM (cf. \cite{Ainsworth04, ib97}) that
\begin{align*}
k-\om=\left\{\begin{array}{ll}
O\big(k^{2p+1}h^{2p}\big) &\text{ if } k h\ll 1, \\
O\Big(\dfrac{k}{p}\Big(\dfrac{e k h}{4p}\Big)^{2p}\Big)&\text{ if } k h\gg 1 \text{ and } 2p+1>\dfrac{e k h}{2}.
\end{array}\right.
\end{align*}
By contrast, our pre-asymptotic error analysis, which works also for \emph{unstructured} meshes, gives error estimates between the exact solution $u$ and the discrete solution $u_h$. Clearly, our pollution error bounds in $H^1$-norm coincide with the phase difference $\abs{k-\om}$ as above.

Our analysis relies on two novel tricks and the profound stability estimates for the continuous problem given in \cite{ms10, ms11}. The first one is a modified duality argument (or Aubin-Nitsche trick). The traditional duality argument is a crucial step in asymptotic error analyses of FEM for scattering problems,
which is usually used to estimate the $L^2$-error of the finite element solution by its $H^1$-error (cf. \cite{ak79,dss94,ib97,ms10,ms11,sch74}).  Our key idea is to use some special designed elliptic projections in the duality-argument step so that we can bound the $L^2$-error of the discrete solution by using the errors of the elliptic projections of the exact solution $u$ and thus obtain the $L^2$-error estimate directly by the modified duality argument. This helps us to derive pre-asymptotic error estimates for the CIP-FEM and the FEM under the condition \eqref{econd1a}. The second trick is used to prove the absolute stability of the CIP-FEM. One of the authors and Feng  \cite{fw09, fw11} have developed an approach to analyze the stability of the interior penalty discontinuous Galerkin methods which makes use of the special test function $v_h:=\na u_h\cdot (x-x_\Om)$ (defined element-wise) and a local version of the Rellich identity (for the Laplacian) and mimics the stability analysis for the PDE solutions given in \cite{melenk95b,cf06,hetmaniuk07}. Here $x_\Om$ is a point such that the domain $\Om$ is strictly star-shaped with respect to it. But the function $v_h$ is not in the approximation space of CIP-FEM which is the same as that of the FEM, i.e., the space of \emph{continuous} piecewise (mapped) polynomials, and thus, it can not server as a test function for the CIP-FEM.  We circumvent this difficulty by using the $L^2$ projection of $v_h$ onto the approximation space of CIP-FEM instead. By using the so-called ``Oswald interpolation" \cite{be07, hw96, kp94, oswald93} we may show that the difference between $v_h$ and its $L^2$ projection is estimated by the jump of $v_h$ at mesh interfaces, and hence, controlled by using the jump term \eqref{ejump} in the CIP-FEM. We remark that the technique for deriving the stability of the linear CIP-FEM developed in the first part of the series \cite{w} does not work for higher order methods because it relies on the fact of $\De u_h=0$ on each element.

The remainder of this paper is organized as follows. The CIP-FEM is introduced in Section~\ref{sec-2}. Some preliminary results,  including the stability of the continuous solution, the approximation properties of the $hp$-finite element space, and estimates of some elliptic projections, are cited or proved in Section~\ref{sec-pre}.  In Section~\ref{sec-preasy-dual}, by using an improved duality argument, some pre-asymptotic error estimates in $H^1$- and $L^2$-norms are given for both CIP-FEM and FEM under the condition that $\big(\frac{k}{p}\big)^{\frac{1}{p+1}}\frac{kh}{p}$ is small enough. In Section~\ref{sec-sta}, the CIP-FEM is shown to be stable for  any $k>0$, $h>0$, $p>0$, and $\ga>0$. In Section~\ref{sec-preasy-sta}, pre-asymptotic error estimates of the CIP-FEM are proved for $k>0$, $h>0$, $p>0$, and $\ga>0$ by utilizing the error estimates for the elliptic projection, the stability results for the CIP-FEM, and the triangle inequality.  
The proofs of two preliminary lemmas on $hp$-approximation properties are given in Appendix~\ref{sec-app}.

Throughout the paper, $C$ is used to denote a generic positive constant
which is independent of $h$, $p$, $k$, $f$, $g$, and the penalty parameters. We also use the shorthand
notation $A\lesssim B$ and $B\gtrsim A$ for the
inequality $A\leq C B$ and $B\geq CA$. $A\simeq B$ is a shorthand
notation for the statement $A\lesssim B$ and $B\lesssim A$. We assume that $k\gtrsim 1$ since we are considering high-frequency problems. For the ease of presentation, we assume that $k$ is constant on the domain $\Om$. We also assume that $\Om$ is a strictly star-shaped domain with an analytic boundary. Here ``strictly star-shaped" means that there exist a point $x_\Om\in\Om$ and a positive constant $c_\Om$ depending only on $\Om$ such that 
\begin{equation}\label{def1}
(x-x_\Om)\cdot n\ge c_\Om,\quad\forall x\in\pa\Om.
\end{equation}
 Sure the theory of this paper can be extended to the case of polyhedral domains. The extensions to other types of boundary conditions, such as PML absorbing boundary condition (cf. \cite{ber94,bw05,cw03}) and DtN boundary condition (cf. \cite{ms10}), will be addressed in future works.

\section{Formulation of CIP-FEM}\label{sec-2}
To formulate our CIP-FEM, we first introduce some notation. The standard space, norm and inner product notation
are adopted. Their definitions can be found in \cite{bs08,ciarlet78}.
In particular, $(\cdot,\cdot)_Q$ and $\pd{ \cdot,\cdot}_\Sigma$
for $\Sigma\subset \pa Q$ denote the $L^2$-inner product
on complex-valued $L^2(Q)$ and $L^2(\Sigma)$
spaces, respectively. Denote by $(\cdot,\cdot):=(\cdot,\cdot)_\Om$
and $\pd{ \cdot,\cdot}:=\pd{ \cdot,\cdot}_{\p\Om}$. 

For the triangulation of $\Om$ we adopt the setting of \cite{ms10, ms11}. The triangulation $\T_h$ consists of elements which are the image of the reference triangle (in two dimensions) or the reference tetrahedron (in three dimensions). We do not allow hanging nodes and assume---as is standard---that the element maps of elements sharing an edge or a face induce the same parametrization on that edge or face. Additionally, the element maps $F_K$ from the reference element $\hat K$ to $K\in \T_h$ satisfy the assumption 5.1 in \cite{ms11}. For any triangle/tetrahedron $K\in \T_h$, we define $h_K:=\mbox{diam}(K)$. 
Similarly, for each edge/face $e$ of $K\in \T_h$, define $h_e:=\mbox{diam}(e)$. Let $h=\max_{K\in\T_h}h_K$.
 We define
\begin{eqnarray*}
\E_h^I&:=& \mbox{ set of all interior edges/faces of $\T_h$},\\
\E_h^B&:=& \mbox{ set of all boundary edges/faces of $\T_h$ on $\Ga$}.
\end{eqnarray*}
We also define the jump $\jm{v}$ of $v$ on an interior edge/face
$e=\p K\cap \p K^\pr$ as
\[
\jm{v}|_{e}:=\left\{\begin{array}{ll}
       v|_{K}-v|_{K^\pr}, &\quad\mbox{if the global label of $K$ is bigger},\\
       v|_{K^\pr}-v|_{K}, &\quad\mbox{if the global label of $K^\pr$ is bigger}.
\end{array} \right.
\]
 For every $e=\p K\cap \p K^\pr\in\E_h^I$, let $n_e$ be the unit outward normal
to edge/face $e$ of the element $K$ if the global label of $K$ is bigger
and of the element $K^\pr$ if the other way around. For every $e\in\E_h^B$, let $n_e=n$ the unit outward normal to $\pa\Om$.

Now we define the ``energy" space $V$ and the sesquilinear
form $a_h(\cdot,\cdot)$ on $V\times V$ as follows:
\begin{align}
V&:=H^1(\Om)\cap\prod_{K\in\T_h} H^2(K), \nonumber \\
\label{eah}
a_h(u,v)&:=(\na u,\na v)+  J(u,v)\qquad\forall\, u, v\in V,
\end{align}
where
\begin{align}
J(u,v):=&\sum_{e\in\E_h^{I}}\i\ga_e\, \frac{h_e}{p^2} \pdjj{\frac{\pa u}{\pa n_e}}{\frac{\pa v}{\pa n_e}},
\label{eJ}\db
\end{align}
and  $\ga_e, e\in\E_h^I$ are 
nonnegative numbers to be specified later. 

\begin{myrem}
(a) The terms in $J(u,v)$ are so-called penalty terms.
The penalty parameter in $J(u,v)$ is
$\i\ga_e$.  So it is a
pure imaginary number with positive imaginary part. It turns out that if it is replaced by a complex number with positive
imaginary part, the ideas of the paper still apply.  Here we set their
real parts to be zero partly because the terms from real parts do not help much
(and do not cause any problem either) in our theoretical analysis and partly for the ease of presentation.

(b) Penalizing the jumps of normal derivatives 
 was used early by Douglas and Dupont \cite{dd76} for second order PDEs and by Babu{\v{s}}ka and Zl\'amal \cite{bz73} for fourth order PDEs in
the context of $C^0$ finite element methods, by Baker \cite{baker77} for fourth order PDEs
and by Arnold \cite{arnold82} for second order
parabolic PDEs in the context of IPDG methods. 

(c) In this paper we consider the scattering problem with time dependence $e^{\i\om t}$, that is, the signs before $\i$'s in the Sommerfeld radiation condition \eqref{e1.2a} and its first order approximation \eqref{eq1.1b} are positive. If we consider the scattering problem with time dependence $e^{-\i\om t}$, that is, the signs before $\i$'s in  \eqref{e1.2a} and  \eqref{eq1.1b} are negative, then the penalty parameters should be complex numbers with  negative imaginary parts.
\end{myrem}

It is clear that $J(u,v)=0$ if $u\in H^2(\Om)$ and $v\in V$. Therefore, if $u\in H^2(\Om)$ is the solution of \eqref{eq1.1a}--\eqref{eq1.1b}, then 
\begin{equation}\label{2.6}
a_h(u,v) - k^2(u,v) +\i k \pd{ u,v}
=(f,v)+\pd{g, v},\qquad\forall v\in V.
\end{equation}

Let $V_h$ be the $hp$-CIP approximation space, that is,
$$V_h:=\set{v_h \in H^1(\Omega):\ v_h\mid_K\circ F_K \in \mathcal P_p(\hat K), \forall K \in \T_h},$$
where $\mathcal P_p(\hat K)$ denote the set of all polynomials whose degrees do not exceed $p$ on $\hat K$. Then our CIP-FEMs are defined as follows: Find $u_h\in V_h$ such that
\begin{equation}\label{ecipfem}
a_h(u_h,v_h) - k^2(u_h,v_h) +\i k \pd{ u_h,v_h}
=(f, v_h)+\pd{g, v_h},  \qquad\forall v_h\in V_h.
\end{equation}
We remark that if the parameters $\ga_e\equiv 0$, then the above CIP-FEM becomes the standard FEM.

The following (semi-)norms on the space $V$ are useful for the subsequent analysis:
\begin{align}\label{e2.5}
\norm{v}_{1,h}:=&\bigg( \norml{\na v}{\Om}^2+\sum_{e\in\E_h^{I}}\ga_e \frac{h_e}{p^2} \norml{\jm{\frac{\pa v}{\pa n_e}}}{e}^2\bigg)^{1/2},\\
 \norme{v}:=&\big(\norm{v}_{1,h}^2+k\norml{v}{\Ga}^2\big)^{\frac12}.\label{e2.5b}
\end{align}

In the next sections, we shall consider the pre-asymptotic stability and
error analysis for the above CIP-FEM and the FEM.  For the ease of presentation, we assume that $\ga_e\simeq \ga$ for some positive constant $\ga$ and that $h_K\simeq h$. 

\section{Preliminary lemmas}\label{sec-pre} In this section, we first recall the stability estimates of the continuous problem from Melenk and Sauter \cite{ms10, ms11}. Then we consider the $hp$-approximation estimates of the discrete space $V_h$. These estimates improve slightly the estimates in \cite{ms11} when $u\in H^2(\Om)$. Finally, we analyze some elliptic projections which are crucial for our pre-asymptotic analysis. 
 
\subsection{Stability estimates of the continuous problem}
Let $\nabla ^n$ stand for derivatives of order $n$; more precisely, for a function $u$: $\Omega\rightarrow \R^d$, $\Omega \subset \R^d$, $|\nabla^n u(x)|^2=\sum_{\alpha \in \mathbb{N}_0^d:|\alpha|=n}\frac{n!}{\alpha !}|D^\alpha u(x)|^2 $. The following lemma (cf. \cite[Theorem 4.10]{ms11}) says that the solution $u$ to the continuous problem \eqref{eq1.1a}--\eqref{eq1.1b} can be decomposed into the sum of an elliptic part and an analytic part $u=\ue+\ua$ where $\ue$ is usually non-smooth but the $H^2$-bound of $\ue$ is independent of $k$ and $\ua$ is oscillatory but the $H^j$-bound of $\ua$ is available for any integer $j\ge 0$.  
\begin{mylem}\label{depcomposition}
The solution $u$ to the problem \eqref{eq1.1a}--\eqref{eq1.1b}
can be written as $u=\ue+\ua$, and satisfies
\begin{align}
\|\ue\|_{H^2(\Om)}+k |\ue|_{H^1(\Om)}+k^2\|\ue\|_{L^2(\Om)}&\leq C C_{f,g},\label{u_E}\\
|u_{\mathcal{A}}|_{H^1(\Om)}+k\|u_{\mathcal{A}}\|_{L^2(\Om)}&\leq C C_{f,g},\label{u_A1}\\
\forall p \in \mathbb{N}_0, \ \ \ \  \|\nabla ^{p+2} u_{\mathcal{A}}\|_{L^2(\Om)}&\leq C \lambda^p k^{-1}\max\{p,k\}^{p+2}
C_{f,g}. \label{u_A2}
\end{align}
Here, $\lambda>1$ independent of $k$ and $C_{f,g}:=\|f\|_{L^2(\Omega)}$ \textnormal{+} $\|g\|_{H^{1/2}(\Gamma)}$.
\end{mylem}
\begin{myrem}\label{ruh2}
An direct consequence of the above lemma is that 
\[\norm{u}_{H^2(\Om)}\lesssim k C_{f,g}.\]
This estimate was proved in \cite{cf06,hetmaniuk07, melenk95b}.
\end{myrem}
\subsection{Approximation properties} In this subsection we consider to approximate the solution $u$ to the problem \eqref{eq1.1a}--\eqref{eq1.1b} by finite element functions in $V_h$. We give two types of approximation estimates in both $L^2$-norm and the norm $\norme{\cdot}$ defined in \eqref{e2.5b}. The first type of estimates can be applied to smooth solution and gives higher order convergences both in $h$ and $p$. The second type of estimates assumes only $H^2$ regularity and gives first order convergences (even for higher order elements), but it uses the decomposition given in Lemma~\ref{depcomposition} and is more subtle than the first one. The proofs are a little long and will be given in the appendix. 
 
\begin{mylem}\label{lapprox1}
Let $s\ge 2$ and $\mu = \min \{p + 1, s\}$. Suppose $u \in H^s(\Omega)$. Then there exists $\hat{u}_h \in V_h$ such that
\begin{align}\label{elapprox1a}
\norml{u-\hat{u}_h}{\Omega}&\lesssim \frac{ h^{\mu}}{p^{s}}\|u\|_{H^s(\Omega)},\\
\norme{u-\hat{u}_h}&\ls \cerr \frac{ h^{\mu-1}}{p^{s-1}}\|u\|_{H^s(\Omega)},\label{elapprox1b}
\end{align}
where $\cerr:=\big(1+\ga+\frac{kh}{p}\big)^{\frac{1}{2}}$.
\end{mylem}

 The following lemma gives approximation estimates of  the solution to the problem \eqref{eq1.1a}-\eqref{eq1.1b}.
 {\begin{mylem}\label{error2}
Let $u$ be the solution to the problem \eqref{eq1.1a}-\eqref{eq1.1b}. Then
there exist $\hat{u}_h \in V_h$ and a constant $\sigma >0$ independent of $k$, $p$, $h$, and the penalty parameters, such that
\begin{align}
\|u-\hat{u}_h\|_{L^2(\Om)}&\lesssim \Big(\frac{h^2}{p^2}+\frac{h}{p^{2}}\Big(\frac{kh}{\sigma p}\Big)^p\Big)C_{f,g},\label{u_err0}\\
 \norme{u-\hat{u}_h}&\lesssim \cerr\Big(\frac{h}{p}+\frac{1}{p}\Big(\frac{kh}{\sigma p}\Big)^p\Big)C_{f,g},\label{u_err2}
\end{align}
where $C_{f,g}:=\|f\|_{L^2(\Omega)}$ \textnormal{+} $\|g\|_{H^{1/2}(\Gamma)}$ and $\cerr:=\big(1+\ga+\frac{kh}{p}\big)^{\frac{1}{2}}$.
\end{mylem}}
\begin{myrem} The above estimates are based on the results from \cite{ms10,ms11} and a duality argument. Similar estimates have been given in \cite{ms10,ms11} as follows
\begin{align*}
\inf_{v_h\in V_h}\big(\norml{\na(u-v_h)}{\Om}+k\norml{u-v_h}{\Om}\big)\ls \Big(1+\frac{k h}{p}\Big)\Big(\frac{h}{p}+\Big(\frac{k h}{\sigma p}\Big)^p\Big)C_{f,g}.
\end{align*} 
Our estimates in this lemma improve a little the above estimate in terms of $p$.  See also Remark~\ref{rem-A1}.
\end{myrem}

\subsection{Elliptic projections} In this subsection, we introduce two kind of elliptic projections and estimate the approximation errors of them. 
For any $u \in  V$, we define its elliptic projections $u_h^+ \in V_h $  and $u_h^- \in V_h $ by the following two formulations, respectively.
\begin{equation}\label{ahuh+}
a_h(u_h^+,v_h)+\i k\langle u_h^+,v_h\rangle =a_h(u,v_h)+\i k\langle u,v_h \rangle \ \ \ \ \ \forall v_h \in V_h,
\end{equation}
\begin{equation}\label{ahuh-}
a_h(v_h,u_h^-)+\i k\langle v_h,u_h^-\rangle =a_h(v_h,u)+\i k\langle v_h,u \rangle \ \ \ \ \ \forall v_h \in V_h.
\end{equation}
Let us take a look at the projections from other points of view. Define
\begin{align*}
a_h^\pm(u,v)&:=(\na u,\na v)\pm\sum_{e\in\E_h^{I}}\i\ga_e\, \frac{h_e}{p^2} \pdjj{\frac{\pa u}{\pa n_e}}{\frac{\pa v}{\pa n_e}}.
\end{align*}
Then $a_h^+=a_h$, $a_h^-(u,v)=\overline{a_h(v,u)}$ (cf. \eqref{eah}), and clearly, $u_h^\pm\in V_h$ satisfies
\begin{equation}\label{ahuh+-}
a_h^\pm(u_h^\pm,v_h)\pm\i k\langle u_h^\pm,v_h\rangle =a_h^\pm(u,v_h)\pm\i k\langle u,v_h \rangle \ \ \ \ \ \forall v_h \in V_h.
\end{equation}
In other words, $u_h^\pm$ is an CIP finite element approximation to the solution $u$ of the following (complex-valued) Poisson problem:
\begin{align*}
-\triangle u&=F \ \ \ \ \text{in} \ \ \Om,\\
\frac{\partial u}{\partial n}\pm\i ku&=\psi \ \ \ \ \text{on} \ \ \Gamma,
\end{align*}
for some given function $F$ and $\psi$ which are determined by $u$.

Next, we estimate the errors of $u_h^\pm$. From \eqref{ahuh+-} we have the following Galerkin orthogonality:
\begin{equation}\label{aheta}
a_h^\pm(u-u_h^\pm,v_h)\pm\i k\langle u-u_h^\pm,v_h\rangle=0 \ \ \ \ \forall v_h \in V_h.
\end{equation}
We state the following continuity and coercivity properties for the sesquilinear form $a_h(\cdot,\cdot)$. Since they follow easily from \eqref{eah}-\eqref{e2.5}, so we omit their proofs to save space.
{\begin{mylem}\label{lemma4.1}
For any $v,w \in V$,
\begin{equation}
|a_h^\pm(v,w)|,|a_h^\pm(w,v)|\leq \|v\|_{1,h}\|w\|_{1,h},
\end{equation}
\begin{equation}
\re a_h^\pm(v,v)\pm \im a_h^\pm(v,v)=\|v\|^2_{1,h}.
\end{equation}
\end{mylem}}
\begin{mylem}\label{error1}
Suppose $u$ is any function in $H^2(\Om)$. Then there hold the following estimates:
\begin{align}\label{lem 2.2}
\norme{u-u_h^\pm}&\lesssim \inf_{z_h \in V_h}\norme{u-z_h},\\
\label{lem 2.3}
\|u-u_h^\pm\|_{L^2(\Om)}&\lesssim \cerr\frac{h}{p}\inf_{z_h \in V_h}\norme{u-z_h},
\end{align}
where $\cerr:=\big(1+\ga+\frac{kh}{p}\big)^{\frac{1}{2}}$.
\end{mylem}
\begin{proof} We only prove the estimates for $u_h^+$ since the proof for $u_h^-$ follows almost the same procedure.
For any $z_h \in V_h$, let $\eta$ =$ u-u_h^+$, $\eta_h$=$u_h^+-z_h$. From $\eta_h+\eta=u-z_h$ and \eqref{aheta}, we have
\begin{equation}\label{lem 2.1}
a_h(\eta,\eta)+\i k\langle\eta,\eta\rangle=a_h(\eta,u-z_h)+\i k\langle \eta,u-z_h \rangle.
\end{equation}
Applying  Lemma~\ref{lemma4.1} and \eqref{lem 2.1} we can obtain that
\begin{align*}
\|\eta\|^2_{1,h}&=\textnormal{Re}\  a_h(\eta,\eta)+\im \  a_h(\eta,\eta)\\
&=\re (a_h(\eta,\eta)+\i k\langle \eta,\eta \rangle)+\im (a_h(\eta,\eta)+\i k\langle \eta,\eta \rangle)-k\langle \eta,\eta \rangle\\
&=\re (a_h(\eta,u-z_h)+\i k\langle \eta,u-z_h \rangle)\\
&\ \ \ \ +\im (a_h(\eta,u-z_h)+\i k\langle \eta,u-z_h \rangle)-k\|\eta\|_{L^2(\Ga)}\\
&\leq C\Big(\|\eta\|_{1,h}\|u-z_h\|_{1,h}+k\|\eta\|_{L^2(\Ga)}\|u-z_h\|_{L^2(\Ga)}\Big)-k\|\eta\|_{L^2(\Ga)}.
\end{align*}
Therefore,
\begin{equation}
\|\eta\|^2_{1,h}+k\|\eta\|^2_{L^2(\Ga)}\lesssim \|u-z_h\|^2_{1,h}+k\|u-z_h\|^2_{L^2(\Ga)}.
\end{equation}
That is, \eqref{lem 2.2} holds.

To show \eqref{lem 2.3}, we use the Nitsche's duality argument (cf. \cite{bs08,ciarlet78}). Consider the following auxiliary problem:
\begin{align}
-\triangle w&=\eta \ \ \ \ \ \text{in} \ \ \Om, \label{auxi1} \\
\frac{\partial w}{\partial n}-\i kw&=0 \ \ \ \ \ \text{on} \ \ \Ga. \label{auxi2}
\end{align}
It can be shown that $w$ satisfies
\begin{equation} \label{auxi11}
\|w\|_{H^2(\Om)}\lesssim \|\eta\|_{L^2(\Om)}.
\end{equation}
As a matter of fact, the PDE theory shows that $\abs{w}_{H^2(\Om)}\le C(\Om)\big(\norml{\De w}{\Om}+\norm{w}_{H^1(\Om)}\big)$ (cf. \cite{ks95}), and testing \eqref{auxi1} by the conjugate of $w$ and taking the real part and imaginary part imply that $\norm{w}_{H^1(\Om)}\ls \norml{\eta}{\Om}$.

Let  $\hat{w}_h \in V_h$  be defined in Lemma~\ref{lapprox1} (with $u$ replaced by $w$). From \eqref{aheta},
\begin{equation*}
a_h(\eta,\hat{w}_h)+\i k\langle \eta,\hat{w}_h\rangle=0.
\end{equation*}
Testing the conjugated \eqref{auxi1} by $\eta $ and using the above orthogonality we get
\begin{align*}
\|\eta\|^2_{L^2(\Om)}&=-(u-u_h^+,\triangle w)=a_h(u-u_h^+,w)+\i k\langle u-u_h^+,w\rangle\\
&=a_h(\eta,w-\hat{w}_h)+\i k\langle \eta,w-\hat{w}_h \rangle\\
&\lesssim \|\eta\|_{1,h}\|w-\hat{w}_h\|_{1,h}+k \|\eta\|_{L^2(\Ga)}\|w-\hat{w}_h\|_{L^2(\Ga)}\\
&\ls\norme{\eta}\norme{w-\hat{w}_h}\\
&\lesssim \norme{\eta}\cerr\frac{h}{p}\|w\|_{H^2(\Om)},
\end{align*}
which together with \eqref{lem 2.2} and \eqref{auxi11} gives \eqref{lem 2.3}. The proof is completed.
\end{proof}

By combining Lemma~\ref{error1} and Lemma~\ref{error2} we have the following lemma which gives the error estimates between the solution to the problem \eqref{eq1.1a}-\eqref{eq1.1b} and its elliptic projections.
\begin{mylem}\label{u-_err}
Let $u$ be the solution to the problem \eqref{eq1.1a}-\eqref{eq1.1b}. Then there hold the following estimates:
\begin{align}
\norme{u-u_h^\pm}& \lesssim \cerr \Big(\frac{h}{p}+\frac{1}{p}\Big(\frac{kh}{\sigma p}\Big)^p\Big)C_{f,g},\\
 \|u-u_h^\pm\|_{L^2(\Om)}& \lesssim \cerr^2\Big(\frac{h^2}{p^2}+\frac{h}{p^2}\Big(\frac{kh}{\sigma p}\Big)^p\Big)C_{f,g}.
\end{align}
where $\cerr:=\big(1+\ga+\frac{kh}{p}\big)^{\frac{1}{2}}$.
\end{mylem}

\section{Pre-asymptotic error estimates by duality argument}\label{sec-preasy-dual} 
One crucial step in asymptotic error analyses  of FEM for scattering problems is performing the 
 duality argument (or Aubin-Nitsche trick) (cf. \cite{ak79,dss94,ib97,ms10,ms11,sch74}). This argument is usually used to estimate the $L^2$-error of the finite element solution by its $H^1$-error.  In this section we use a modified duality argument to derive pre-asymptotic error estimates under the condition that $\big(\frac{k}{p}\big)^{\frac{1}{p+1}}\frac{kh}{p}$ is sufficiently small. The key idea is to used the elliptic projections from the previous section in the duality-argument step so that we can bound the $L^2$-error of the discrete solution by using the errors of the elliptic projections of the exact solution $u$ and thus obtain the $L^2$-error estimate directly by the modified duality argument.

\begin{mythm}\label{thm-err-1}
Let $u$ and $u_h$ denote the solutions of \eqref{eq1.1a}-\eqref{eq1.1b} and \eqref{ecipfem},respectively. Suppose $0\le \ga\ls 1$. Then there exist constants $C_0$ and $\sigma > 0$ independent of $k, h, p$, and the penalty parameters, such that if 
\begin{equation}\label{econd1}
\frac{kh}{p}\le C_0\Big(\frac{p}{k}\Big)^{\frac{1}{p+1}}, 
\end{equation}
then the following error estimates hold:
\begin{align}\label{error-eh1}
\norme{u-u_h}&\lesssim \Big(1+\frac{k}{p}\Big(\frac{kh}{\sigma p}\Big)^p\Big)\inf_{z_h \in V_h}\norme{u-z_h},\\
\|u-u_h\|_{L^2(\Om)}&\lesssim \Big(\frac{h}{p}+\frac{1}{p}\Big(\frac{kh}{\sigma p}\Big)^p\Big)\inf_{z_h \in V_h}\norme{u-z_h}.\label{error-eh2}
\end{align}
\end{mythm}
\begin{proof} 
First, we estimate the $L^2$-error by introducing the dual problem and using the elliptic projections of the solution to the original continuous problem and of the solution to the dual problem. 
Let $e_h:=u-u_h$. Consider the following dual problem:
\begin{align}
-\triangle w-k^2w&=e_h \ \ \ \ \rm{in} \ \ \Om, \label{auxii1}\\
\frac{\partial w}{\partial n}-\i kw&=0 \ \ \ \ \ \rm{on} \ \ \Ga.
\end{align}
Let $u_h^+$ be the elliptic projection of $u$ defined by \eqref{ahuh+} and let $w_h^-\in V_h$ be the elliptic projection defined as  \eqref{ahuh-}. From \eqref{2.6} and \eqref{ecipfem} we have the following orthogonality,
\begin{equation}\label{ecipfemorth}
a_h(e_h,v_h)-k^2(e_h,v_h)+\i k\langle e_h,v_h \rangle=0,\quad\forall v_h\in V_h.
\end{equation}
Testing the conjugated \eqref{auxii1} by $e_h$ and using \eqref{ecipfemorth} with $v_h=w_h^-$, \eqref{ahuh+}--\eqref{ahuh-}  we get
\begin{align}\label{eduall2}
\|e_h\|^2_{L^2(\Om)}&=(\nabla e_h,\nabla w)-k^2(e_h,w)+\i k\langle e_h,w\rangle\\
&=a_h(e_h,w)-k^2(e_h,w)+\i k\langle e_h,w\rangle\nn\\
&=a_h(e_h,w-w_h^-)+\i k\langle e_h,w-w_h^- \rangle-k^2(e_h,w-w_h^-)\nn\\
&=a_h(u-u_h^+,w-w_h^-)+\i k\langle u-u_h^+,w-w_h^- \rangle-k^2(e_h,w-w_h^-).\nn
\end{align}
Similar to Lemma~\ref{u-_err}, we may show that
\begin{align}
\norme{w-w_h^-}& \ls \cerr \Big(\frac{h}{p}+\frac{1}{p}\Big(\frac{kh}{\sigma p}\Big)^p\Big)\norml{e_h}{\Om},\label{ewh-err1}\\
 \|w-w_h^-\|_{L^2(\Om)}& \lesssim \cerr^2\Big(\frac{h^2}{p^2}+\frac{h}{p^2}\Big(\frac{kh}{\sigma p}\Big)^p\Big)\norml{e_h}{\Om}.\label{ewh-err2}
\end{align}
From \eqref{eduall2}, Lemma~\ref{lemma4.1}, Lemma~\ref{error1}, and \eqref{ewh-err1}--\eqref{ewh-err2}, we have
\begin{align*}
\|e_h\|^2_{L^2(\Om)} &\le\|u-u_h^+\|_{1,h}\|w-w_h^-\|_{1,h}+k\norml{u-u_h^+}{\Ga}\norml{w-w_h^-}{\Ga} \\
&\ \ \ \ \ \ \ \ \ \ \ \ \ \ \ \ +k^2\norml{e_h}{\Om}\norml{w-w_h^-}{\Om}\\
&\le C \cerr\Big(\frac{h}{p}+\frac{1}{p}\Big(\frac{kh}{\sigma p}\Big)^p\Big)\norml{e_h}{\Om}\inf_{z_h \in V_h}\norme{u-z_h}\\
&\ \ \ \ \ \ \ \ \ \ \ \ \ \ \ \ +C \cerr^2k^2\Big(\frac{h^2}{p^2}+\frac{h}{p^2}\Big(\frac{kh}{\sigma p}\Big)^p\Big)\norml{e_h}{\Om}^2. \\
\end{align*}
Noting that $\cerr=1+\ga+\frac{kh}{p}\ls 1+\frac{kh}{p}$, there exists a constant $C_0$ independent of $k, h, p$, and the penalty parameters such that 
\begin{equation}\label{econd1-a}\text{ if } \frac{kh}{p}\le C_0\Big(\frac{p}{k}\Big)^{\frac{1}{p+1}} \text{ then }C \cerr^2k^2\Big(\frac{h^2}{p^2}+\frac{h}{p^2}\Big(\frac{kh}{\sigma p}\Big)^p\Big)\le C \cerr^2\Big(C_0^2\Big(\frac{p}{k}\Big)^{\frac{2}{p+1}}+\Big(\frac{C_0}{\sigma}\Big)^{p+1}\Big)\le \frac12,\end{equation}
and as a consequence,
\begin{equation}\label{error-e-h-o}
\|e_h\|_{L^2(\Om)}\ls  \Big(\frac{h}{p}+\frac{1}{p}\Big(\frac{kh}{\sigma p}\Big)^p\Big)\inf_{z_h \in V_h}\norme{u-z_h}.
\end{equation}
That is, \eqref{error-eh2} holds. 

Next we turn to prove \eqref{error-eh1}.  Let $u_h^\pm\in V_h$ be the elliptic projections of $u$ defined by \eqref{ahuh+}--\eqref{ahuh-}  and denote by $\zeta^\pm=u-u_h^\pm$. It follows from Lemma~\ref{lemma4.1} and \eqref{ecipfemorth} that
\begin{align*}
\|e_h\|^2_{1,h}=&\re a_h(e_h,e_h)+\im a_h(e_h,e_h)\\
 =&\re (a_h(e_h,e_h)-k^2(e_h,e_h)+\i k\langle e_h,e_h \rangle)+k^2(e_h,e_h)\\
 &+\im(a_h(e_h,e_h)-k^2(e_h,e_h)+\i k\langle e_h,e_h \rangle)-k\|e_h\|^2_{L^2(\Ga)}\\
 =&\re (a_h(e_h,\zeta^-)-k^2(e_h,\zeta^-)+\i k\langle e_h,\zeta^- \rangle)+k^2\|e_h\|^2_{L^2(\Om)}\\
 &+\im(a_h(e_h,\zeta^-)-k^2(e_h,\zeta^-)+\i k\langle e_h,\zeta^- \rangle)-k\|e_h\|^2_{L^2(\Ga)}\\
 =&\re (a_h(\ze^+,\zeta^-)-k^2(e_h,\zeta^-)+\i k\langle \ze^+,\zeta^- \rangle)+k^2\|e_h\|^2_{L^2(\Om)}\\
 &+\im(a_h(\ze^+,\zeta^-)-k^2(e_h,\zeta^-)+\i k\langle \ze^+,\zeta^- \rangle)-k\|e_h\|^2_{L^2(\Ga)}\\
 \leq& 2\|\ze^+\|_{1,h}\|\ze^-\|_{1,h}+k^2\|\ze^-\|_{L^2(\Om)}^2+2k\|\ze^+\|_{L^2(\Ga)}\|\ze^-\|_{L^2(\Ga)}
 +2k^2\|e_h\|^2_{L^2(\Om)}-k\|e_h\|^2_{L^2(\Ga)}.
 \end{align*}
Therefore, from Lemma~\ref{error1} and noting that $\frac{kh}{p}, \ga\ls 1$,
\begin{align*}
\norme{e_h}&\ls \inf_{z_h \in V_h}\norme{u-z_h}+k\|e_h\|_{L^2(\Om)}.
\end{align*}
which together with \eqref{error-eh2} implies \eqref{error-eh1}. This completes the proof of the theorem.
\end{proof}
\begin{myrem}\label{rthm-err-1}
{\rm (i)} Noting that if $\ga=0$ then the CIP-FEM becomes the FEM, the above theorem and the three corollaries below hold for the standard FEM. 

{\rm (ii)} Noting that the wave length is $\frac{2\pi}{k}$, the condition \eqref{econd1} roughly says that about $\frac{2\pi}{C_0}\big(\frac{k}{p}\big)^{\frac{1}{p+1}}$ degrees of freedom are needed in one wave length.

{\rm (iii)} From \eqref{econd1-a}, clearly,  the theorem holds under the following condition which is a little more general (and a little more complicated as well) than \eqref{econd1}:
 \begin{equation}\label{econd1-b}
\frac{k^2 h^2}{p^2}+\frac{k}{p}\Big(\frac{kh}{\sigma p}\Big)^{p+1} \text{ is sufficiently small.}
\end{equation}

{\rm (iv)} If
\begin{align}
&\frac{k h}{p}+\frac{k}{p}\Big(\frac{kh}{\sigma p}\Big)^p \text{ is sufficiently small,}\label{econd-asy}\end{align}
 then CIP-FEM, as well as the FEM, is pollution free in the norm $\norme{\cdot}$, i.e., $\norme{u-u_h}\lesssim \inf_{z_h \in V_h}\norme{u-z_h}.$
 The condition \eqref{econd-asy} improves a little the condition of $\frac{k h}{p}+k\big(\frac{kh}{\sigma p}\big)^p$ small enough given in \cite{ms10} for the FEM, because, for example, if we choose $k\big(\frac{kh}{\sigma p}\big)^p=c$ for some fixed constant $c$ so that the latter condition does not hold, then $\frac{k h}{p}+\frac{k}{p}\big(\frac{kh}{\sigma p}\big)^p=\big(\frac{c}{k}\big)^{\frac1p}+\frac{c}{p}$ may be small enough for an appropriate choice of $p$ and large $k$. 

{\rm (v)} The pre-asymptotic error estimates for the linear CIP-FEM and FEM ($p=1$) in two and three dimensions are given in the first part of this series \cite{w}. For the pre-asymptotic error estimates for the linear and $hp$ version of the FEM in one dimension, we refer to \cite{ib95a} and \cite{ib97}, respectively. To the best of the authors' knowledge, there have been no pre-asymptotic error estimates for the higher order FEM and CIP-FEM ($p\ge 2$) in two and three dimensions given in the literature so far. 
\end{myrem}

From Remark~\ref{rthm-err-1}(iv) we have the following corollary which gives practical sufficient conditions for asymptotic estimates of the CIP-FEM and the FEM.
\begin{mycor}\label{cor-asy}
Let $u$ and $u_h$ denote the solutions of \eqref{eq1.1a}-\eqref{eq1.1b} and \eqref{ecipfem},respectively. Suppose $0\le \ga\ls 1$. Then the following asymptotic error estimate 
$$\norme{u-u_h}\lesssim \inf_{z_h \in V_h}\norme{u-z_h}$$ holds under any of the following three conditions:
\begin{align}
&\Big(\frac{k}{p}\Big)^{\frac{1}{p}}\frac{kh}{p} \text{ is sufficiently small} , \label{econd-asy-b}\\
&p\gtrsim (\ln k)^{\frac12} \text{ and } \Big(\frac{k}{p}\Big)^{\frac{1}{p+1}}\frac{kh}{p} \text{ is sufficiently small},\label{econd-asy-c}\\
&p\gtrsim \ln k  \text{ and } \frac{k\,h}{p} \text{ is sufficiently small}.\label{emscond1}
\end{align}
\end{mycor}
\begin{proof}\eqref{econd-asy-b} implies \eqref{econd-asy} and the conditions \eqref{econd-asy-c}--\eqref{emscond1} are all sufficient conditions of \eqref{econd-asy-b}.
\end{proof}
\begin{myrem}
{\rm (i)} Our asymptotic estimates hold for both the CIP-FEM and the FEM. Melenk and Sauter \cite{ms10,ms11} proved recently the asymptotic error estimates for the FEM under either of the condition \eqref{emscond1} and the following one:
\begin{align}
&p=O(1) \text{ fixed independent of } k\text{ and } k^{1+\frac1p}h \text{ is sufficiently small}.\label{emscond2}
\end{align}
The above condition is similar to the condition \eqref{econd-asy-b} but we prefer the later one because it does not require $p=O(1)$.

{\rm (ii)} The condition \eqref{econd-asy-c} says that the theorem~{\rm\ref{thm-err-1}} gives actually asymptotic estimates in the norm $\norme{\cdot}$ under the further condition of $p\gtrsim (\ln k)^{\frac12}$. This condition is not strict from the practical point of view. For example if we take $p\simeq (\ln k)^{\frac12}$ then $p$ grows slowly as $k$ increases, and so is $\big(\frac{k}{p}\big)^{\frac{1}{p+1}}\simeq k^{\frac{1}{\sqrt{\ln k}}}$ which is $o(k^\ep)$ for any positive constant $\ep$.

{\rm (iii)} The condition \eqref{econd-asy-b} is convenient for fixed $p$. The condition \eqref{emscond1} requires $p\gtrsim \ln k$ and is perfect when implementations of sufficiently high order CIP-FEM or FEM  are available. And the condition \eqref{econd-asy-c} which requires merely $p\gtrsim (\ln k)^{\frac12}$ is suitable when only ``medium" order methods can be used.  
\end{myrem}

From Theorem~\ref{thm-err-1} and Lemma~\ref{error2}, we have the following corollary which gives estimates for $H^2$ regular solutions.
\begin{mycor}\label{cor-1}
Suppose the solution $u\in H^2(\Om)$. Under the conditions of Theorem~\ref{thm-err-1}, there hold the following estimates:
\begin{align}\label{ecor-1-a}
\norm{u-u_h}_{1,h}&\lesssim\Big(\frac{h}{p}+\frac{1}{p}\Big(\frac{kh}{\sigma p}\Big)^p+\frac{k}{p^2}\Big(\frac{kh}{\sigma p}\Big)^{2p}\Big)C_{f,g},\\
\|u-u_h\|_{L^2(\Om)}&\lesssim \Big(\frac{h^2}{p^2}+\frac{1}{p^2}\Big(\frac{kh}{\sigma p}\Big)^{2p}\Big)C_{f,g}.\label{ecor-1-b}
\end{align}
\end{mycor}
\begin{myrem} {\rm (i)} If $p=o\big((\ln k)^{\frac12}\big)$ then the condition \eqref{econd1} is weaker then the condition \eqref{econd-asy-b} for asymptotic error estimates, and the estimate \eqref{ecor-1-a} is a pre-asymptotic one with pollution term $O\Big(\frac{k}{p^2}\big(\frac{kh}{\sigma p}\big)^{2p}\Big)$.  

{\rm (ii)} Pre-asymptotic error analysis and dispersion analysis are two main tools to understand numerical behaviors in short wave computations. The later one which is usually performed on \emph{structured} meshes estimates the error between the wave number $k$ of the continuous problem and some discrete wave number $\om$ \cite{Ainsworth04, Ainsworth09, dbb99, harari97, ib95a, ib97, thompson94, thompson06}. In particular, it is shown for the $hp$-FEM (cf. \cite{Ainsworth04, ib97}) that
\begin{align*}
k-\om=\left\{\begin{array}{ll}
O\big(k^{2p+1}h^{2p}\big) &\text{ if } k h\ll 1, \\
O\Big(\dfrac{k}{p}\Big(\dfrac{e k h}{4p}\Big)^{2p}\Big)&\text{ if } k h\gg 1 \text{ and } 2p+1>\dfrac{e k h}{2}.
\end{array}\right.
\end{align*}
By contrast, our pre-asymptotic error analysis gives the error between the exact solution $u$ and the discrete solution $u_h$ and works for \emph{unstructured} meshes.  Clearly, our pollution error bounds in $H^1$-norm coincide with the phase difference $\abs{k-\om}$ as above.

{\rm (iii)} Our pre-asymptotic estimates in Theorem~\ref{thm-err-1} and Corollary~\ref{cor-1}--\ref{cor-2}  hold in particular under the following condition
\begin{equation}\label{econd2}
p=O(1) \text{ fixed independent of } k\text{ and } k^{1+\frac{1}{p+1}}h \text{ is small enough},
\end{equation} 
which extends the asymptotic range given by \eqref{emscond2}. 
\end{myrem}

By combining Theorem~\ref{thm-err-1} and Lemma~\ref{lapprox1}, we have the following corollary which gives estimates for $H^s$ regular solutions $(s>2)$.
\begin{mycor}\label{cor-2}
Suppose the solution $u\in H^s(\Om)$, $s>2$. Let $\mu = \min \{p + 1, s\}$. Under the conditions of Theorem~\ref{thm-err-1}, there hold the following estimates:
\begin{align}\label{ecor-2-a}
\norm{u-u_h}_{1,h}&\lesssim\Big(1+\frac{k}{p}\Big(\frac{kh}{\sigma p}\Big)^p\Big)\frac{ h^{\mu-1}}{p^{s-1}}\|u\|_{H^s(\Omega)},\\
\|u-u_h\|_{L^2(\Om)}&\lesssim \Big(\frac{h}{p}+\frac{1}{p}\Big(\frac{kh}{\sigma p}\Big)^p\Big)\frac{ h^{\mu-1}}{p^{s-1}}\|u\|_{H^s(\Omega)},\label{ecor-2-b}
\end{align}
where the invisible constants are independent on $p, h, k$, and the penalty parameters, but may depend on $s$. In particular, if $u\in H^{p+1}(\Om)$ is an oscillating solution in the sense of \cite[Definition 3.2]{ib97}, i.e.
\[\norm{u}_{H^{p+1}(\Om)}\ls k^p,\]
then the following estimates hold:
\begin{align}\label{ecor-2-c}
\norm{u-u_h}_{1,h}&\lesssim C(p)\Big(\Big(\frac{kh}{\sigma p}\Big)^p+\frac{k}{p}\Big(\frac{kh}{\sigma p}\Big)^{2p}\Big),\\
\|u-u_h\|_{L^2(\Om)}&\lesssim \frac{C(p)}{k}\Big(\Big(\frac{kh}{\sigma p}\Big)^{p+1}+\frac{k}{p}\Big(\frac{kh}{\sigma p}\Big)^{2p}\Big).\label{ecor-2-d}
\end{align}
\end{mycor}
\begin{myrem} {\rm (i)} Corollary~\ref{cor-2} certainly holds for $s=2$. We exclude the the case $s=2$ in this corollary because the estimates in \eqref{ecor-2-a}--\eqref{ecor-2-b} with $s=2$ are worse than those in \eqref{ecor-1-a}--\eqref{ecor-1-b}. For $s>2$, Corollary~\ref{cor-2} gives optimal convergence order in $h$ and $p$ while Corollary~\ref{cor-1} gives only first order convergence in $h$ and $p$.

{\rm (ii)} If $p=O(1)$ and $u$ is an oscillating solution then \eqref{ecor-2-c} shows clearly that the pollution error in $H^1$-norm is again bounded by $O(k^{2p+1}h^{2p})$.
\end{myrem}

By combining Lemma~\ref{depcomposition} and Corollary~\ref{cor-1} we have the following stability estimates for the CIP-FEM (FEM).
\begin{mycor}\label{cor-3}
Suppose the solution $u\in H^2(\Om)$. Under the conditions of Theorem~\ref{thm-err-1}, there hold the following estimates:
\begin{align*}
\norm{u_h}_{1,h}+k\norml{u_h}{\Om}&\lesssim C_{f,g},
\end{align*}
and hence the CIP-FEM is well-posed.
\end{mycor}
\begin{myrem} If $\ga=0$, then this corollary gives the following stability estimate for the standard FEM under the same condition.
 \begin{align*}
\norml{\na u_h}{\Om}+k\norml{u_h}{\Om}\lesssim C_{f,g}.
\end{align*}
\end{myrem}

\section{Stability estimates for the CIP-FEM}\label{sec-sta} 
In the previous section we have shown that the CIP-FEM and the FEM are stable under the condition that $\frac{kh}{p}\le C_0\big(\frac{p}{k}\big)^{\frac{1}{p+1}}$ and $0\le\ga\ls 1$. The goal of this section is to derive stability estimates (or a priori estimates) for the CIP-FEM \eqref{ecipfem} for any $k, h, p$, and $\ga>0$.

Some stability estimates have been proved for the discontinuous Galerkin methods (cf. \cite{fw09,fw11, fx})
and the spectral-Galerkin methods (cf. \cite{sw07}). Their analyses mimic the stability analysis for the PDE solutions given in
\cite{cf06,hetmaniuk07,melenk95b}. The key idea is to use the test functions $v=u_h$ and $v=(x-x_{\Omega})\cdot \nabla u_h$, respectively, and use the Rellich identity(for the Laplacian), where $x_{\Omega}$ is a point such that the domain $\Omega$ is strictly star-shaped with respect to it (see \eqref{def1}). As for our CIP-FEM \eqref{ecipfem}, although the test
function $v_h=u_h$ can be still be used, the test function $v_h=(x-x_{\Omega})\cdot \nabla u_h$ does not apply
since it is discontinuous and hence not in the test space $V_h$. We surround this difficulty by taking the $L^2$ projection of $(x-x_{\Omega})\cdot \nabla u_h$ to the finite element space $V_h$. To analyze the error of the $L^2$ projection, we recall the so-called Oswald interpolation operator $\ios $,
which has been analyzed in \cite{be07, hw96, kp94, oswald93}. Let $W_h$ be the space of (discontinuous) piecewise mapped polynomials: 
\begin{align}\label{eW_h}
W_h:=\prod\limits_{K \in \mathcal{T}_h }\mathcal{P}_p(F_K^{-1}(K)).
\end{align}
\begin{mylem} \label{Ios}
There exist an operator $\ios : W_h\mapsto V_h$ and a constant $C_p$ depending only on $p$, such that, for all $K \in \mathcal{T}_h$, the following estimate holds:
\begin{align}
&\forall v_h \in W_h, \ \ \ \|v_h-\ios  v_h\|_{L^2(K)}\ls C_p h_K^{\frac{1}{2}}\sum_{e\in \mathcal E_h^I\cap K}\|[v_h]\|_{L^2(e)},
\end{align}
where $C_p=p^{-1}$ for Cartesian meshes and $C_p=p^{\frac{d-3}{2}}$ for triangulations $(d=2,3)$.
\end{mylem}

\begin{mylem} \label{lem-L2}
Let $Q_h$ be the $L^2$ projection onto $V_h$ and let $C_p$ be the constant in Lemma~\ref{Ios}. Then the following estimate holds:
\begin{align}
&\forall v_h \in W_h, \ \ \ \|v_h-Q_h v_h\|_{L^2(\T_h)}\ls C_p \bigg(\sum_{e\in \E_h^I}h_e\|[v_h]\|_{L^2(e)}^2\bigg)^{\frac12},\label{eL2a}\\
&\forall v_h \in  W_h, \ \ \ \|\nabla(v_h-Q_h  v_h)\|_{L^2(\T_h)}\ls \frac{p ^2 C_p}{h} \bigg(\sum_{e\in \E_h^I}h_e\|[v_h]\|_{L^2(e)}^2\bigg)^{\frac12}.\label{eL2b}
\end{align}
\end{mylem}
\begin{proof}
\eqref{eL2a} follows from $\|v_h-Q_h v_h\|_{L^2(\T_h)}\le \|v_h-\ios v_h\|_{L^2(\T_h)}$ and Lemma~\ref{Ios}. \eqref{eL2b} follows from the inverse inequality and \eqref{eL2a}.
\end{proof}
 
We cite the following lemma \cite[Lemma 4.1]{fw09},  which establishes two integral identities and play a crucial role in our analysis.
{\begin{mylem}
Let $\alpha(x):=x-x_{\Omega}$, $v \in \prod \limits_{K \in \mathcal T_h}H^2(K)$, $K\in \mathcal T_h$ and $e \in \mathcal E_h^I$. Then there hold
\begin{align}
&d \|v\|^2_{L^2(K)}+2\textnormal{Re}(v,\alpha \cdot\nabla v)_K=\int_{\partial K}\alpha\cdot n_K|v|^2, \label{Rellich 1}\\
&(d-2)\|\nabla v\|^2_{L^2(K)}+2\textnormal{Re}(\nabla v,\nabla(\alpha \cdot\nabla v))_K=\int_{\partial K}\alpha\cdot n_K|\nabla v|^2. \label{Rellich 2}
\end{align}
Here $x_{\Om}$ is a point such that the domain $\Omega$ is strictly star-shaped with respect to it (see \eqref{def1}).
\end{mylem}}
\begin{myrem}
The identity \eqref{Rellich 2} can be viewed as a local version of the Rellich identity for the Laplacian $\triangle$ (cf. \cite{cf06}).
\end{myrem}

\begin{mylem}
Let $u_h \in V_h $ solve \eqref{ecipfem}. Then
\begin{align}
& \norml{\na u_h}{\Om}^2-k^2\|u_h\|^2_{L^2(\Om)}\leq 2\norml{f}{\Om}\norml{u_h}{\Om}+\frac1k\norml{g}{\Ga}^2,\label{Re1} \\
& \sum_{e\in \mathcal E_h ^I}\gamma_e \frac{h_e}{p^2}\norm{\jm{\frac{\partial u_h}{\partial n_e}}}^2_{L^2(e)}+k \|u_h\|^2_{L^2(\Gamma)}\leq 2\norml{f}{\Om}\norml{u_h}{\Om}+\frac1k\norml{g}{\Ga}^2.\label{Im1}
\end{align}
\end{mylem}
\begin{proof}
Taking $v_h=u_h$ in \eqref{ecipfem} yields
\begin{align} \label{ahuh}
a_h(u_h,u_h)-k^2 \|u_h\|^2_{L^2(\Om)}+\i k\|u_h\|^2_{L^2(\Gamma)}=(f,u_h)+\langle g,u_h\rangle.
\end{align}
Therefore, taking real part and imaginary part of the above equation, we get
\begin{align}
& \norml{\na u_h}{\Om}^2-k^2\|u_h\|^2_{L^2(\Om)}\leq |(f,u_h)+\langle g,u_h\rangle|,\label{Re1a} \\
& \sum_{e\in \mathcal E_h ^I}\gamma_e \frac{h_e}{p^2}\norm{\jm{\frac{\partial u_h}{\partial n_e}}}^2_{L^2(e)}+k \|u_h\|^2_{L^2(\Gamma)}\leq |(f,u_h)+\langle g,u_h\rangle|.\label{Im1a}
\end{align}
From \eqref{Im1a},
\begin{align*}
\sum_{e\in \mathcal E_h ^I}\gamma_e \frac{h_e}{p^2}\norm{\jm{\frac{\partial u_h}{\partial n_e}}}^2_{L^2(e)}+k \|u_h\|^2_{L^2(\Gamma)}&\leq |(f,u_h)+\langle g,u_h\rangle|\\
&\le \norml{f}{\Om}\norml{u_h}{\Om}+\frac{1}{2k}\norml{g}{\Ga}^2+\frac{k}{2} \|u_h\|^2_{L^2(\Gamma)}
\end{align*}
which implies \eqref{Im1}. From \eqref{Re1a} and \eqref{Im1},
\begin{align*}
 \norml{\na u_h}{\Om}^2-k^2\|u_h\|^2_{L^2(\Om)}&\leq \norml{f}{\Om}\norml{u_h}{\Om}+\frac{1}{2k}\norml{g}{\Ga}^2+\frac{k}{2} \|u_h\|^2_{L^2(\Gamma)}\\
&\le 2\norml{f}{\Om}\norml{u_h}{\Om}+\frac1k\norml{g}{\Ga}^2.
\end{align*}
That is, \eqref{Re1} holds. This completes the proof of the lemma.
\end{proof}

 From \eqref{Re1} and \eqref{Im1} we can bound $\norml{\na u_h}{\Om}$ and the jumps of $\frac{\partial u_h}{\partial n_e}$ across $e \in \mathcal E_h^I$.
In order to get the desired a priori estimates, we need to derive a reverse inequality whose coefficients can be controlled. Such a reverse inequality, which is often difficult to get under practical mesh constraints, and stability estimates for $u_h$ will be derived next.
{\begin{mythm}\label{stablity}
Let $u_h \in V_h$ solve \eqref{ecipfem} and suppose $ \gamma_e\simeq \gamma$, $h_K,h_e\simeq h$. Then
\begin{equation}\label{||u_h||_{1,h}}
k\|u_h\|_{L^2(\Om)}+\|u_h\|_{1,h}\lesssim \csta M(f,g),
\end{equation}
where
\begin{align}
 M(f,g)&:=\|f\|_{L^2(\Om)}+\|g\|_{L^2(\Gamma)}, \label{M(f,g)}\\
 \csta&:=1+\frac{1}{k}\Big(\frac{\ga p^4}{h^2}+\frac{p^6 C_p^2}{\ga h^2}\Big).\label{c_{sta}}
\end{align}
\end{mythm}}
\begin{proof}
We divide the proof into three steps.

\ \ \ \textit{Step 1: Derivation of a representation identity for $\|u_h\|_{L^2(\Omega)}$}. Define $v_h$ by $v_h|_K=\alpha\cdot\nabla u_h|_K$ for every $K\in \mathcal T_h$, denote by $w_h=Q_h  v_h$, hence $w_h\in V_h$. Using this $w_h$ as a test function in \eqref{ecipfem} and taking the real part of resulted equation we get
\begin{equation}\label{eqawh}
-k^2 \re (u_h,w_h)=\re ((f,w_h)+\langle g,w_h \rangle-a_h(u_h,w_h)-\i k\langle u_h,w_h\rangle).
\end{equation}
It follows from \eqref{Rellich 1}, \eqref{ahuh}, and \eqref{eqawh} that
\begin{align}\label{th1}
&2k^2\|u_h\|^2_{L^2(\Om)}= k^2\sum_{K \in \mathcal T_h}\int_{\partial K}\alpha\cdot n_K|u_h|^2-(d-2)k^2\|u_h\|^2_{L^2(\Om)} \\
\nn&\qquad \ \ \ \ \ \ \ \ \ \ \ \ \ \  -2k^2\re (u_h,w_h) -2k^2\re (u_h,v_h-w_h) \\
\nn &=k^2\sum_{K \in \mathcal T_h}\int_{\partial K}\alpha\cdot n_K|u_h|^2+(d-2)\re ((f,u_h)+\langle g,u_h\rangle-a_h(u_h,u_h))\\
\nn & \quad +2 \re ((f,w_h)+\langle g,w_h \rangle-a_h(u_h,w_h)-\i k \langle u_h,w_h \rangle)\\
\nn &= k^2\sum_{K \in \mathcal T_h}\int_{\partial K}\alpha\cdot n_K|u_h|^2+(d-2)\re ((f,u_h)+\langle g,u_h\rangle)\\
\nn &\quad +2 \re ((f,w_h)+\langle g,w_h \rangle)-\sum_{K\in \mathcal T_h}((d-2)\|\nabla u_h\|^2_{L^2(K)}+2\re (\nabla u_h,\nabla w_h)_K)\\
\nn &\quad -2\re J(u_h,w_h)+2k\im \langle u_h,w_h \rangle.
\end{align}
Since $u_h$ is continuous, we have
\begin{align}\label{th22}
\sum_{K \in \mathcal T_h}\int_{\partial K}\alpha\cdot n_K|u_h|^2&=2\sum_{e\in \mathcal E_h^I}\textnormal{Re}\langle \alpha \cdot n_e\{u_h\},[u_h]\rangle_e
+\langle \alpha \cdot n_\Omega,|u_h|^2 \rangle\\
 \nn &=\langle \alpha \cdot n_\Omega,|u_h|^2 \rangle.
\end{align}
Using the identity $|a|^2-|b|^2=$ \rm{Re}$(a+b)(\bar{a}-\bar{b})$ followed by the Rellich identity \eqref{Rellich 2} we get
\begin{align}\label{th2}
&\sum_{K\in \mathcal T_h}((d-2)\|\nabla u_h\|^2_{L^2(K)}+2\re (\nabla u_h,\nabla w_h)_K)\\
\nn&\ \  =\sum_{K\in \mathcal T_h}((d-2)\|\nabla u_h\|^2_{L^2(K)}+2\re (\nabla u_h,\nabla v_h)_K)\\
\nn & \ \ \ \ \ \ \ \ \ \ \ \ \ \ \ \ \ \ \ \ \ \ \ \ +2\sum_{K \in \mathcal T_h}\re (\nabla u_h,\nabla(w_h-v_h))_K\\
\nn &\ \ =\sum_{K\in \mathcal T_h}\int_{\partial K}\alpha\cdot n_K |\nabla u_h|^2+2\sum_{K \in \mathcal T_h}\re (\nabla u_h,\nabla(w_h-v_h))_K\\
\nn &\ \ =2\sum_{e\in \mathcal E_h^I}\re \langle \alpha\cdot n_e\{\nabla u_h\},[\nabla u_h]\rangle_e+\langle \alpha\cdot n_\Om,|\nabla u_h|^2\rangle\\
\nn &\ \ \ \ \ \ \ \ \ \ \ \ \ \ \ \ \ \ \ \ \ \ \ \ \ \ \ \  +2\sum_{K \in \mathcal T_h}\re (\nabla u_h,\nabla(w_h-v_h))_K.
\end{align}
Plugging \eqref{th22} and \eqref{th2} into \eqref{th1} gives
\begin{align}\label{th3}
2k^2\|u_h\|^2_{L^2(\Om)}&=(d-2)\re ((f,u_h)+\langle g,u_h\rangle)+2 \re ((f,v_h)+\langle g,v_h \rangle)\\
\nn &\quad +k^2\langle \alpha\cdot n_\Om,|u_h|^2\rangle+2k\im \langle u_h,v_h \rangle-\langle \alpha\cdot n_\Om,|\nabla u_h|^2\rangle\\
\nn &\quad -2\sum_{e\in \mathcal E_h^I}\re \langle \alpha\cdot n_e\{\nabla u_h\},[\nabla u_h]\rangle_e-2\re J(u_h,v_h)\\
\nn &\quad -2\re J(u_h,w_h-v_h)-2\sum_{K \in \mathcal T_h}\re (\nabla u_h,\nabla(w_h-v_h))_K\\
\nn&\quad+2k\im \langle u_h,w_h-v_h \rangle+2 \re ((f,w_h-v_h)+\langle g,w_h-v_h \rangle).
\end{align}
\ \ \ \textit{Step 2: Derivation of a reverse inequality}. We bound each term on the right-hand side of \eqref{th3}. 
We have 
\begin{align}\label{esta1}
2\re ((f,v_h)+\langle g,v_h \rangle) \le& C\|f\|^2_{L^2(\Om)}+\frac{1}{8}\norml{\na u_h}{\Om}^2\\&\nn+C \|g\|^2_{L^2(\Ga)}+\frac{c_\Om}{4}\sum_{e \in \mathcal E^B_h}\|\nabla u_h\|^2_{L^2(e)}.
\end{align}
It is clear that
\begin{align}
k^2\langle \alpha\cdot n_\Om,|u_h|^2\rangle\leq C k^2 \|u_h\|^2_{L^2(\Ga)}.
\end{align}
It follows from the star-shaped assumption on $\Omega$ that
\begin{align}
& 2k\im \langle u_h,v_h \rangle-\langle \alpha\cdot n_\Om,|\nabla u_h|^2\rangle\\
\nn &\leq C k\sum_{e\in \mathcal E_h^B}\|u_h\|_{L^2(e)}\|\nabla u_h\|_{L^2(e)}-c_\Om\sum_{e \in \mathcal E_h^B}\|\nabla u_h\|^2_{L^2(e)}\\
\nn &\leq Ck^2\|u_h\|^2_{L^2(\Gamma)}-\frac{c_\Om}{2}\sum_{e\in\mathcal E_h^B}\|\nabla u_h\|^2_{L^2(e)}.
\end{align}
For an edge/face $e \in \mathcal E_h^I$, let $K_e$ and $K'_e$ denote the two elements in $\mathcal T_h$ that share $e$. We have
\begin{align}
-2\sum_{e\in \mathcal E_h^I}\re \langle \alpha\cdot n_e\{\nabla u_h\},[\nabla u_h]\rangle_e &\leq C\sum_{e\in \mathcal E_h^I}\Big(\frac{p^2}{h}\Big)^{\frac12}\|\nabla u_h\|_{L^2(K_e\bigcup K'_e)}\norm{\jm{\frac{\partial u_h}{\partial n_e}}}_{L^2(e)}\\
\nn & \leq \frac{1}{8}\norml{\na u_h}{\Om}^2+C\frac{p^4}{\gamma h^2}\sum_{e \in \mathcal E_h^I}\ga_e \frac{h_e}{p^2}\norm{\jm{\frac{\partial u_h}{\partial n_e}}}^2_{L^2(e)}.
\end{align}
From \eqref{eJ}, the trace and inverse inequalities (cf. Lemma~\ref{trace}), we have
\begin{align} \label{Juhvh}
-2\re J(u_h,v_h)& =-2\re\sum_{e \in \mathcal E_h^I}\i\ga_e\frac{h}{p^2}\Big\langle\jm{\frac{\partial u_h}{\partial n_e}},\jm{\frac{\partial v_h}{\partial n_e}}\Big\rangle_e\\
& \ls \sum_{e \in \mathcal E_h^I}\ga_e \frac{h_e}{p^2}\norml{\jm{\frac{\partial u_h}{\partial n_e}}}{e}\norm{\jm{\frac{\partial v_h}{\partial n_e}}}_{L^2(e)}\nn\\
\nn &\lesssim \sum_{e \in \mathcal E_h^I}\ga_e \frac{h_e}{p^2}\norm{\jm{\frac{\partial u_h}{\partial n_e}}}_{L^2(e)}\Big(\frac{p^2}{h}\Big)^{\frac12}\sum_{K=K_e,K'_e}|v_h|_{H^1(K)}\\
\nn & \lesssim \sum_{e \in \mathcal E_h^I}\ga_e \frac{h_e}{p^2}\norm{\jm{\frac{\partial u_h}{\partial n_e}}}_{L^2(e)}\Big(\frac{p^2}{h}\Big)^{\frac12}\sum_{K=K_e,K'_e}\frac{p^2}{h}\|\nabla u_h\|_{L^2(K)}\\
\nn & \leq \frac{1}{8}\norml{\na u_h}{\Om}^2+\frac{\ga p^4}{h^2}\sum_{e \in \mathcal E_h^I}\ga_e\frac{h_e}{p^2}\norm{\jm{\frac{\partial u_h}{\partial n_e}}}^2_{L^2(e)}.
\end{align}
Next we estimate the terms containing $v_h-w_h=v_h-Q_h v_h$. From Lemma \ref{lem-L2} and the definition of $v_h$,\begin{align}
&\norml{v_h-w_h}{\T_h}\ls \frac{p C_p}{\ga^{\frac12}}  \bigg(\sum_{e\in \mathcal E_h^I}\ga_e \frac{h_e}{p^2} \norm{\jm{\frac{\partial u_h}{\partial n_e}}}_{L^2(e)}^2\bigg)^{\frac12},\label{Ios1}\\
&\abs{v_h-w_h}_{H^1(\T_h)}\ls \frac{p^3 C_p}{\ga^{\frac12} h}  \bigg(\sum_{e\in \mathcal E_h^I}\ga_e \frac{h_e}{p^2} \norm{\jm{\frac{\partial u_h}{\partial n_e}}}_{L^2(e)}^2\bigg)^{\frac12}.\label{Ios2}
\end{align}
From \eqref{Ios2} and the trace inequality in Lemma~\ref{trace},
\begin{align}\label{esta5}
-2\re J(u_h,w_h-& v_h)\leq\sum_{e\in \mathcal E_h^I}\ga_e \frac{h_e}{p^2}\norm{\jm{\frac{\partial u_h}{\partial n_e}}}_{L^2(e)}\|[\nabla(w_h-v_h) ]\|_{L^2(e)}\\
\nn &\leq C\sum_{e\in \mathcal E_h^I}\ga_e \frac{h_e}{p^2}\norm{\jm{\frac{\partial u_h}{\partial n_e}}}_{L^2(e)}\sum_{K=K_e,K'_e}\Big(\frac{p^2}{h}\Big)^\frac{1}{2}\|\nabla(w_h-v_h) \|_{L^2(K)}\nn\\
\nn &\leq C\ga^{\frac12}\bigg(\sum_{e\in \mathcal E_h^I}\ga_e \frac{h_e}{p^2}\norm{\jm{\frac{\partial u_h}{\partial n_e}}}_{L^2(e)}^2\bigg)^{\frac12}\abs{v_h-w_h}_{H^1(\T_h)}\nn\\
\nn & \leq C\frac{p^3C_p}{h}\sum_{e \in \mathcal E_h^I}\ga_e\frac{h_e}{p^2}\norm{\jm{\frac{\partial u_h}{\partial n_e}}}^2_{L^2(e)}.\nn
\end{align}
And we have,
\begin{align}
-2\sum_{K \in \mathcal T_h}\re (\nabla u_h,\nabla(w_h-&v_h))_K  \leq 2\sum_{K\in \mathcal T_h}\|\nabla u_h\|_{L^2(K)}\|\nabla(w_h-v_h)\|_{L^2(K)}\\
\nn &\leq \frac{1}{8}\norml{\na u_h}{\Om}^2+C \abs{v_h-w_h}_{H^1(\T_h)}^2\\
\nn &\leq \frac{1}{8}\norml{\na u_h}{\Om}^2+C \frac{p^6 C_p^2}{\ga h^2} \sum_{e\in \mathcal E_h^I}\ga_e \frac{h_e}{p^2}\norm{\jm{\frac{\partial u_h}{\partial n_e}}}_{L^2(e)}^2.
\end{align}
From \eqref{Ios1} and Lemma~\ref{trace},
\begin{align}
2k\im \langle u_h,w_h-v_h \rangle &\le C k\sum_{e\in \E_h^B}\norml{u_h}{e}\norml{w_h-v_h}{e} \\
&\le C k\norml{u_h}{\Ga}\Big(\frac{p^2}{h}\Big)^{\frac12}\norml{w_h-v_h}{\T_h}\nn\\
&\le C k\norml{u_h}{\Ga}\Big(\frac{p^2}{h}\Big)^{\frac12}\frac{p C_p}{\ga^{\frac12}}  \bigg(\sum_{e\in \mathcal E_h^I}\ga_e \frac{h_e}{p^2}\norm{\jm{\frac{\partial u_h}{\partial n_e}}}_{L^2(e)}^2\bigg)^{\frac12}\nn\\
 &\le C p^2C_p\Big(\frac{k }{\ga h}\Big)^{\frac12}\bigg(k\norml{u_h}{\Ga}^2+\sum_{e \in \mathcal E_h^I}\ga_e \frac{h_e}{p^2}\norm{\jm{\frac{\partial u_h}{\partial n_e}}}^2_{L^2(e)}\bigg).\nn
 \end{align}
From \eqref{Ios1}, \eqref{Ios2}, and Lemma~\ref{trace},
\begin{align}\label{esta10}
2& \re ((f,w_h-v_h)+\langle g,w_h-v_h \rangle)\\
&\lesssim \norml{f}{\Om}\norml{w_h-v_h}{\T_h}+\norml{g}{\Ga}\Big(\frac{p^2}{h}\Big)^{\frac12}\norml{w_h-v_h}{\T_h}\nn\\
\nn &\ls \big(\norml{f}{\Om}+\norml{g}{\Ga}\big)\frac{p^2 C_p}{(\ga h)^{\frac12}}  \bigg(\sum_{e\in \mathcal E_h^I}\ga_e \frac{h_e}{p^2}\norm{\jm{\frac{\partial u_h}{\partial n_e}}}_{L^2(e)}^2\bigg)^{\frac12}\\
\nn &\ls \norml{f}{\Om}^2+\norml{g}{\Ga}^2+ \frac{p^4 C_p^2}{\ga h}\sum_{e\in \mathcal E_h^I}\ga_e \frac{h_e}{p^2}\norm{\jm{\frac{\partial u_h}{\partial n_e}}}_{L^2(e)}^2.
\end{align}
Plugging \eqref{esta1}--\eqref{Juhvh} and \eqref{esta5}--\eqref{esta10} into \eqref{th3} we obtain
\begin{align*}
2k^2\|u_h\|^2_{L^2(\Om)}\leq &\abs{(f,u_h)+\langle g,u_h\rangle}+C(\|f\|^2_{L^2(\Om)}+\|g\|^2_{L^2(\Ga)})\\
\nn &+\frac{1}{2}\norml{\na u_h}{\Om}^2-\frac{c_\Om}{4}\sum_{e \in \mathcal E_h^B}\|\nabla u_h\|^2_{L^2(e)}+C \Big(k+p^2C_p\Big(\frac{k }{\ga h}\Big)^{\frac12}\Big) k\|u_h\|^2_{L^2(\Gamma)}\\
\nn & +C\Big(\frac{p^4}{\gamma h^2}+\frac{\ga p^4}{h^2}+\frac{p^3C_p}{h}+\frac{p^6 C_p^2}{\ga h^2}+p^2C_p\Big(\frac{k }{\ga h}\Big)^{\frac12} \Big)\sum_{e \in \mathcal E_h^I}\ga_e \frac{h_e}{p^2}\norm{\jm{\frac{\partial u_h}{\partial n_e}}}^2_{L^2(e)}.
\end{align*}
From \eqref{Im1} and the facts
\begin{align*}
C_p\ge p^{-1},\quad \frac{p^3C_p}{h}\ls \ga+ \frac{p^6 C_p^2}{\ga h^2},\quad p^2C_p\Big(\frac{k }{\ga h}\Big)^{\frac12}\ls \frac{p^6 C_p^2}{\ga h^2}+\frac{k h}{p^2},
\end{align*}
 we have
\begin{align*}
& 2k^2\|u_h\|^2_{L^2(\Om)}+\frac{c_\Om}{4}\sum_{e \in \mathcal E_h^B}\|\nabla u_h\|^2_{L^(e)}\\
& \quad \leq -\frac{1}{2}\norml{\na u_h}{\Om}^2+\norml{\na u_h}{\Om}^2+C(\|f\|^2_{L^2(\Om)}+\|g\|^2_{L^2(\Ga)})\\
& \quad\quad\  +C\Big(k+1+\frac{\ga p^4}{h^2}+\frac{p^6 C_p^2}{\ga h^2}+\frac{k h}{p^2}\Big)\Big(\norml{f}{\Om}\norml{u_h}{\Om}+\frac1k\norml{g}{\Ga}^2\Big).
\end{align*}
\textit{Step 3: Finishing up}. It follows from \eqref{Re1}, \eqref{M(f,g)}, \eqref{c_{sta}}, and the above inequality that
\begin{align*}
& 2k^2\|u_h\|^2_{L^2(\Om)}+\frac{c_\Om}{4}\sum_{e \in \mathcal E_h^B}\|\nabla u_h\|^2_{L^2(e)}\\
&\leq  k^2\|u_h\|^2_{L^2(\Om)}-\frac{1}{2}\norml{\na u_h}{\Om}^2+C (\|f\|^2_{L^2(\Om)}+\|g\|^2_{L^2(\Ga)})\\
&\quad +C k \csta \big(\norml{f}{\Om}\norml{u_h}{\Om}+\frac1k\norml{g}{\Ga}^2\big)\\
&\leq \frac{3 k^2}{2}\|u_h\|^2_{L^2(\Om)}-\frac{1}{2}\norml{\na u_h}{\Om}^2+C \csta^2 M(f,g)^2,
\end{align*}
which together with \eqref{Im1} implies \eqref{||u_h||_{1,h}}. The proof is completed.
\end{proof}

Since scheme \eqref{ecipfem} is a linear complex-valued system, an immediate consequence of the stability estimates is the following well-posedness theorem for \eqref{ecipfem}.
{\begin{mythm}\label{thh2}
The CIP-FEM \eqref{ecipfem} has a unique solution for $k>0$, $h>0$, $p\geq1$ and $\gamma >0$.
\end{mythm}}
\begin{myrem}\label{r5.3}
{\rm (i)} For the general case when the meshes may be nonuniform, Theorem~\ref{stablity} and Theorem~\ref{thh2} still
hold with $h$ replaced by $\underline{h}=\min_{K\in \mathcal{T}_h}h_K$. The proof is similar and is omitted.

{\rm (ii)} When $p=1$, a better stability estimate is derived in \cite{w} by another approach taking advantage of $\De u_h=0$ on each element.

{\rm (iii)} The stability may be enhanced by an over-penalized technique, just like what is done for the $hp$-IPDG method (cf. \cite{fw11}). More precisely, if we replace $a_h(u_h,v_h)$ in the CIP-FEM \eqref{ecipfem} by the following sesquilinear 
\begin{align*}
a_h^q(u_h,v_h)=(\na u_h, \na v_h)+\sum_{j=1}^q\sum_{e\in\E_h^{I}}\i\ga_{j,e} \left(\frac{h_e}{p^2}\right)^{2j-1}\pdjj{\frac{\p^j u_h}{\p n_e^j}}{\frac{\p^j v_h}{\p n_e^j}},
\end{align*}
with  $\ga_{j,e}\simeq \ga_j>0$, $j=1,2,\cdots,q$, $2\le q\le p$, then, after some tedious but similar derivations, we may show that \eqref{||u_h||_{1,h}} holds with $\csta$ replaced by the following stability constant:
\begin{align*}
\csta^q=1&+\frac{1}{k}\Big(\frac{p^4C_p^2}{\ga_1
h}+\frac{p^4C_p^2}{(\ga_1\ga_2)^{\frac12} h} +
\frac{p^3C_p}{h}\Big(\sum\limits_{j=1}^q\frac{\ga_j}{\ga_1}\Big)^{\frac12}+\frac{p^2}{h}\max\limits_{1\leq j \leq
q-1}\Big(\frac{\ga_j}{\ga_{j+1}}\Big)^{\frac12}+\ga_q \frac{p^4}{h^2}\Big).
\end{align*}
We omit the details here. Clearly, if we choose $
 \ga_j\simeq (C_p^2h)^{\frac{2(j-1)}{2q-1}}, 1\le j\le q,
$ 
then
$\csta^q\ls \dfrac{p^4C_p^{\frac{4q-4}{2q-1}}}{k h^{\frac{2q}{2q-1}}}$, where $C_p$ is defined in Lemma~\ref{Ios}. Recall that the best stability constant obtained so far for the  $hp$-IPDG method is $O\Big(\dfrac{p^{\frac83}}{kh^{\frac43}}\Big)$  (cf. \cite[Theroem~3.3]{fw11}). 
\end{myrem}

\section{Pre-asymptotic error estimates for the CIP-FEM by using the stability}\label{sec-preasy-sta}
In this subsection we shall derive error estimates for scheme (\ref{ecipfem}). This will be done by exploiting the linearity of the Helmholtz equation and making use of the stability estimates derived in Theorem~\ref{stablity} and the projection error estimates established in Lemma~\ref{u-_err}.

Let $u$ and $u_h$ denote the solutions of \eqref{eq1.1a}-\eqref{eq1.1b} and \eqref{ecipfem}, respectively.  Recall that $e_h=u-u_h$. 
Let $u_h^+$ be the elliptic projection of $u$ as defined in \eqref{ahuh+}. Write $e_h=\eta-\xi$ with $\eta:=u-u_h^+,\xi:=u_h-u_h^+$. From the Galerkin orthogonality \eqref{ecipfemorth} and \eqref{ahuh+} we get
\begin{align}
a_h(\xi,v_h)-k^2(\xi,v_h)+\i k\langle \xi,v_h \rangle_{\Ga}&=a_h(\eta,v_h)-k^2(\eta,v_h)+\i k\langle \eta,v_h \rangle_{\Ga}\\
\nn &= -k^2(\eta,v_h)\ \ \ \ \forall v_h \in V_h.
\end{align}
The above equation implies that $\xi \in V_h$ is the solution of scheme \eqref{ecipfem} with sources terms $f=-k^2 \eta$ and $g\equiv 0$. Then an application of Theorem~\ref{stablity} and Lemma~\ref{error1} immediately gives the following lemma.
\begin{mylem}\label{u-u_h err}
$\xi=u_h-u_h^+$ satisfies the following estimate:
\begin{equation}
k\|\xi\|_{L^2(\Om)}+\|\xi\|_{1,h}\lesssim \csta  \cerr\frac{k^2h}{p}\inf_{z_h \in V_h}\norme{u-z_h}.
\end{equation}
\end{mylem}

We are ready to state our error estimate results for scheme \eqref{ecipfem}, which follows from Lemma~\ref{error1}, Lemma~\ref{u-u_h err} and an application of the triangle inequality.
\begin{mythm}\label{u-u_h}
Let $u$ and $u_h$ denote the solution of \eqref{eq1.1a}-\eqref{eq1.1b} and \eqref{ecipfem}, respectively. Suppose $u\in H^s(\Omega), s\ge 2$. Then
\begin{align}
&\|u-u_h\|_{1,h} \lesssim \Big(1+ \cerr\csta \frac{k^2h}{p}\Big)\inf_{z_h \in V_h}\norme{u-z_h},\\
&\|u-u_h\|_{L^2(\Om)} \lesssim C_{\rm{err}}\Big(1+k \csta\Big)\frac{h}{p}\inf_{z_h \in V_h}\norme{u-z_h},
\end{align}
where $\cerr:=\big(1+\ga+\frac{kh}{p}\big)^{\frac{1}{2}}$ and $\csta$ is defined in Theorem~\ref{stablity}.
\end{mythm}
\begin{myrem} 
{\rm (i)} For the linear CIP-FEM, we have shown that \cite{w}, if $\ga\simeq 1$, then the following pre-asymptotic error estimate hold for any $k, h>0$:
\begin{align*}
\|u-u_h\|_{1,h}\le \big(C_1 k h + C_2\min\set{k^3h^2, 1}\big) C_{f,g}.
\end{align*}

{\rm (ii)} The theorem can also be extended to the over-penalized CIP-FEM. 
In particular (cf. Remark~\ref{r5.3} (iii)), if $p\ge q\ge 2$, and $\ga_j\simeq (C_p^2h)^{\frac{2(j-1)}{2q-1}}, 1\le j\le q,$ 
then the following pre-asymptotic error estimate holds for any $k, h, p>0$: 
\begin{align*}
\|u-u_h\|_{1,h} \le C_1 \inf_{z_h \in V_h}\norme{u-z_h}+C_2\Big(1+\frac{kh}{p}\Big)^{\frac{1}{2}} k p^3C_p^{\frac{4q-4}{2q-1}}h^{\frac{-1}{2q-1}}\inf_{z_h \in V_h}\norme{u-z_h}.
\end{align*}
The details will be reported in a separate work.
\end{myrem}


\appendix
\section{Approximation by $hp$-finite element}\label{sec-app}
In this appendix, we prove Lemma~\ref{lapprox1} and Lemma~\ref{error2}.

\subsection{Proof of Lemma~\ref{lapprox1}}
We need the following trace and inverse inequalities (cf. \cite{schwab98,be07}).
\begin{mylem} \label{trace}
For any $K \in \mathcal T_h$ and $z\in \mathcal P_p(F_K^{-1}(K))$,
\begin{align*}
\|z\|_{L^2(\partial K)}&\lesssim ph^ {-\frac{1}{2}}\|z\|_{L^2(K)},\\
\|\nabla z\|_{L^2(K)}&\lesssim p^2h^{-1}\|z\|_{L^2(K)}.
\end{align*}
\end{mylem}
We recall the following well-known $hp$ approximation properties (cf. \cite{bs87,guo06,gs07,ms10}):
\begin{mylem}\label{u_u_E}
Let $\mu = \min \{p + 1, s\}$. Suppose $u \in H^s(\Omega)$.

$\bullet$ There exists $\check{u}_h \in W_h$ such that, 
\begin{align}\label{eapprox1}
\|u-\check{u}_h\|_{H^j(\mathcal{T}_h)}:=\Big(\sum_{K\in \mathcal{T}_h}\|u-\check{u}_h\|^2_{H^2(K)}\Big)^{\frac{1}{2}}
 \lesssim \frac{ h^{\mu-j}}{p^{s-j}}\|u\|_{H^s(\Omega)},\quad j=0,1,2.
\end{align}

$\bullet$ There exists $\hat{u}_h \in V_h$ such that
\begin{align}\label{eapprox2}
\|u-\hat{u}_h\|_{H^j(\Omega)}\lesssim \frac{ h^{\mu-j}}{p^{s-j}}\|u\|_{H^s(\Omega)},\ \ \ j=0,1.
\end{align}
Here the invisible constants in the two inequalities above depend on $s$ but are independent of $k, h$, and $p$.
\end{mylem}

Clearly,  \eqref{elapprox1a} holds (cf. \eqref{eapprox2}). It remains to prove \eqref{elapprox1b}. It follows from the \eqref{eapprox2} and the trace inequality that
\begin{equation}\label{eA1}
\norml{u-\hat{u}_h}{\Ga}\lesssim\norml{u-\hat{u}_h}{\Om}^{\frac12}\norm{u-\hat{u}_h}_{H^1(\Om)}^{\frac12}\ls \frac{ h^{\mu-\frac12}}{p^{s-\frac12}}\|u\|_{H^s(\Omega)}.
\end{equation}
From the local trace inequality 
\begin{equation}\label{eA2}
\|v\|^2_{L^2(\partial K)}\lesssim h_K^{-1}\|v\|^2_{L^2(K)}+\|v\|_{L^2(K)}\|\nabla v\|_{L^2(K)},
\end{equation}
we have
\begin{align*}
\sum_{K \in \mathcal{T}_h}\sum_{e \in \partial K}\norm{\frac{\partial (u-\hat{u}_h)}{\partial n_e}}^2_{L^2(e)}
 \lesssim h^{-1}\|u-\hat{u}_h\|^2_{H^1(\mathcal{T}_h)}+\|u-\hat{u}_h\|_{H^1(\mathcal{T}_h)}\|u-\hat{u}_h\|_{H^2(\mathcal{T}_h)}.
\end{align*}
On the other hand, from Lemma~\ref{trace},
\begin{align*}
\|u-\hat{u}_h\|_{H^2(\mathcal{T}_h)}&\leq \|u-\check{u}_h\|_{H^2(\mathcal{T}_h)}+\|\check{u}_h-\hat{u}_h\|_{H^2(\mathcal{T}_h)}\\
\nn &\lesssim \|u-\check{u}_h\|_{H^2(\mathcal{T}_h)}+\frac{p^2}{h}\|\check{u}_h-\hat{u}_h\|_{H^1(\mathcal{T}_h)}
\end{align*}
Therefore, from \eqref{eapprox1} and \eqref{eapprox2},
\begin{align*}
\sum_{K \in \mathcal{T}_h}\sum_{e \in \partial K}\norm{\frac{\partial (u-\hat{u}_h)}{\partial n_e}}^2_{L^2(e)}
\ls \frac{p^2}{h}\frac{ h^{2\mu-2}}{p^{2s-2}}\|u\|_{H^s(\Omega)}^2.
\end{align*}
Then it follows from \eqref{eapprox2} and \eqref{e2.5} that
\begin{align*}
\|u-\hat{u}_h\|^2_{1,h}&\ls \norml{\na(u-\hat{u}_h)}{\Om}^2+\sum_{e \in \mathcal{E}_h^I}\ga_e \frac{h_e}{p^2}\norm{\jm{\frac{\partial (u-\hat{u}_h)}{\partial n_e}}}^2_{L^2(e)}\\
\nn &\lesssim (1+\gamma)\frac{ h^{2\mu-2}}{p^{2s-2}}\|u\|_{H^s(\Omega)}^2.
\end{align*}
By combining the above estimate, \eqref{eA1}, and \eqref{e2.5b} we conclude that \eqref{elapprox1b} holds. This completes the proof of Lemma~\ref{lapprox1}.

\subsection{Proof of Lemma~\ref{error2}}
Lemma~\ref{u_u_E} can also be used to approximate the elliptic part of the solution $u$. The following lemma gives an approximation to the analytic part of the solution $u$. 
\begin{mylem}\label{u_u_A1}
Suppose $\ua$ is analytic and satisfies \eqref{u_A1}--\eqref{u_A2}.

$\bullet$ There exists $\check{u}_{\mathcal{A},h}\in W_h$ such that
\begin{align}\label{u_A-pi_p}
\|u_{\mathcal{A}}-\check{u}_{\mathcal{A},h}\|_{H^j(\mathcal{T}_h)}&:=\Big(\sum_{K\in \mathcal{T}_h}\|u_{\mathcal{A}}-\check{u}_{\mathcal{A},h}\|^2_{H^2(K)}\Big)^{\frac{1}{2}}\\
\nn & \lesssim \Big\{\frac{h^{2-j}}{k p^{2-j}}+\frac{h^{1-j}}{p}\Big(\frac{kh}{\sigma p}\Big)^p\Big\}C_{f,g},\ \ \ j=0,1,2.
\end{align}

$\bullet$ There exists $\hat{u}_{\mathcal{A},h} \in V_h$ such that
\begin{align}\label{u_A-pi}
\|u_{\mathcal{A}}-\hat{u}_{\mathcal{A},h}\|_{H^j(\Om)}\lesssim \Big\{\frac{h^{2-j}}{k p^{2-j}}+\frac{h^{1-j}}{p^{2-j}}\Big(\frac{kh}{\sigma p}\Big)^p\Big\}C_{f,g}, \ \ \  j=0,1.
\end{align}
Here $\sigma > 0$ is some constant independent of $k, h,$ and $p$.
\end{mylem}
\begin{myrem}\label{rem-A1} The following estimate is proved in \cite[Theorem 5.5]{ms10} (see also \cite[Proposition~5.3]{ms11}):
\begin{align*}
\inf_{v_h\in V_h}\big(\norml{\na(\ua-v_h)}{\Om}+k\norml{\ua-v_h}{\Om}\big)\ls \Big(1+\frac{k h}{p}\Big)\Big(\frac{h}{p}+\Big(\frac{k h}{\sigma p}\Big)^p\Big).
\end{align*}
Clearly, a combination of our $H^1$-estimate and $L^2$-estimate in \eqref{u_A-pi} gives better upper bound than the above estimate. This is because our $L^2$-estimate improves that implied in the proof of \cite[Theorem 5.5]{ms10}.
\end{myrem}
\begin{proof}
Following the proof of \cite[Theorem 5.5]{ms10}, we start by defining for each element $K \in \mathcal{T}_h$ the constant $C_K$ by
\begin{equation}
C^2_K:=\sum_{p \in \mathbb{N}_0}\frac{\|\nabla^p u_{\mathcal{A}}\|^2_{L^2(K)}}{(2\lambda \max\{p,k\})^{2p}}
\end{equation}
 and we note
\begin{align}
\|\nabla^p u_{\mathcal{A}}\|_{L^2(K)}& \leq (2\lambda \max\{p,k\})^{p}C_K \ \ \ \ \forall p \in \mathbb{N}_0,\\
\sum_{K\in\mathcal{T}_h}C^2_K &\lesssim \frac{1}{k^2}C^2_{f,g}. \label{C_K}
\end{align}
Then, \cite[Lemma C.2]{ms10} and a scaling argument provides an approximant $\check{u}_{\mathcal{A},h} \in W_h$ which satisfies for $j=0,1,2$,
\begin{equation*}
\|u_{\mathcal{A}}-\check{u}_{\mathcal{A},h}\|_{H^j(K)}\lesssim h^{-j}C_K\Big\{\Big(\frac{h}{h+\sigma}\Big)^{p+1}+\Big(\frac{kh}{\sigma p}\Big)^{p+1}\Big\}.
\end{equation*}
Summation over all elements $K \in \mathcal{T}_h$ gives
\begin{align*}
\|u_{\mathcal{A}}-\check{u}_{\mathcal{A},h}\|^2_{H^j(\mathcal{T}_h)}&:=\Big(\sum_{K\in \mathcal{T}_h}\|u_{\mathcal{A}}-\check{u}_{\mathcal{A},h}\|^2_{H^j(K)}\Big)\\
&\lesssim h^{-2j}\Big[\Big(\frac{h}{h+\sigma}\Big)^{2p+2}+\Big(\frac{kh}{\sigma p}\Big)^{2p+2}\Big]\sum_{K \in \mathcal{T}_h}C^2_K .
\end{align*}
The combination of the above inequality and \eqref{C_K} yields
\begin{equation}
\|u_{\mathcal{A}}-\check{u}_{\mathcal{A},h}\|_{H^j(\mathcal{T}_h)}\lesssim \Big[\frac{1}{k}h^{-j}\Big(\frac{h}{h+\sigma}\Big)^{p+1}+\frac{h^{1-j}}{p}\Big(\frac{kh}{\sigma p}\Big)^{p}\Big]C_{f,g}.
\end{equation}
Furthermore, we estimate using $h \leq$ diam $\Omega $ and $\sigma > 0$ (independent of $h$)
\begin{align}\label{p-h-s}
h^{-j}\Big(\frac{h}{h+\sigma}\Big)^{p+1}&=\frac{h^{2-j}}{p^{2-j}}\frac{h^{p-1}}{(h+\sigma)^{p+1}}p^{2-j}
\leq\frac{h^{2-j}}{p^{2-j}}\Big(\frac{h}{h+\sigma}\Big)^{p-1}\frac{p^2}{\sigma ^2} \\
\nn &\leq\frac{h^{2-j}}{p^{2-j}}\Big(\frac{\rm{diam}\  \Om}{\sigma+\rm{diam}\  \Om}\Big)^{p-1}\frac{p^2}{\sigma ^2}\lesssim \frac{h^{2-j}}{p^{2-j}},
\end{align}
we therefore arrive at
\begin{equation*}
\|u_{\mathcal{A}}-\check{u}_{\mathcal{A},h}\|_{H^j(\mathcal{T}_h)}\lesssim \Big[\frac{1}{k}\frac{h^{2-j}}{p^{2-j}}+\frac{h^{1-j}}{p}\Big(\frac{kh}{\sigma p}\Big)^{p}\Big]C_{f,g}.
\end{equation*}
That is, \eqref{u_A-pi_p} holds.

Similarly, \cite[Lemma C.3]{ms10} and a scaling argument provides an approximant $u_{\mathcal{A},h} \in V_h$ which satisfies for $j=1$,
\begin{equation*}
\|u_{\mathcal{A}}-u_{\mathcal{A},h}\|_{H^1(K)}\lesssim h^{-1}C_K\Big\{\Big(\frac{h}{h+\sigma}\Big)^{p+1}+\Big(\frac{kh}{\sigma p}\Big)^{p+1}\Big\}.
\end{equation*}
Summation over all elements $K \in \mathcal{T}_h$ gives
\begin{equation*}
\|u_{\mathcal{A}}-u_{\mathcal{A},h}\|^2_{H^1(\Om)}\lesssim \Big[h^{-2}\Big(\frac{h}{h+\sigma}\Big)^{2p+2}+\frac{k^2}{p^2}\Big(\frac{kh}{\sigma p}\Big)^{2p}\Big]\sum_{K \in \mathcal{T}_h}C^2_K .
\end{equation*}
The combination of the above inequality and \eqref{C_K}, \eqref{p-h-s} yields
\begin{equation*}
\|u_{\mathcal{A}}-u_{\mathcal{A},h}\|_{H^1(\Om)}\lesssim \Big[\frac{h}{k\, p}+\frac{1}{p}\Big(\frac{kh}{\sigma p}\Big)^{p}\Big]C_{f,g}.
\end{equation*}
We introduce the sesquilinear form $a(\cdot,\cdot):=(\nabla u, \nabla v)+(u,v)$ and define the elliptic projection $\hat{u}_{\mathcal{A},h} \in V_h$ of $\ua$ by
\begin{equation}\label{elli}
a(\hat{u}_{\mathcal{A},h},v_h)=a(\ua,v_h) \ \ \ \ \forall v_h  \in V_h.
\end{equation}
Then,
\begin{align*}
\|u_{\mathcal{A}}-\hat{u}_{\mathcal{A},h}\|^2_{H^1(\Om)}&=a(u_{\mathcal{A}}-\hat{u}_{\mathcal{A},h},u_{\mathcal{A}}-\hat{u}_{\mathcal{A},h})
=a(u_{\mathcal{A}}-\hat{u}_{\mathcal{A},h},u_{\mathcal{A}}-u_{\mathcal{A},h})\\
\nn &\lesssim \|u_{\mathcal{A}}-\hat{u}_{\mathcal{A},h}\|_{H^1(\Om)}\|u_{\mathcal{A}}-u_{\mathcal{A},h}\|_{H^1(\Om)}.
\end{align*}
Therefore,
\begin{equation}\label{u_A_H^1}
\|u_{\mathcal{A}}-\hat{u}_{\mathcal{A},h}\|_{H^1(\Om)}\lesssim \|u_{\mathcal{A}}-u_{\mathcal{A},h}\|_{H^1(\Om)}\lesssim \Big[\frac{h}{k\, p}+\frac{1}{p}\Big(\frac{kh}{\sigma p}\Big)^{p}\Big]C_{f,g}.
\end{equation}
To show \eqref{u_A-pi}, we use the Nitsche's duality argument and consider the following auxiliary problem:
\begin{align}
-\triangle w+w&=u_{\mathcal{A}}-\hat{u}_{\mathcal{A},h} \ \ \ \ \ \text{in} \ \ \Om, \label{eqauxi1} \\
\frac{\partial w}{\partial n}&=0 \ \ \ \ \ \ \ \ \ \ \ \ \ \ \ \ \ \text{on} \ \ \Ga.
\end{align}
It can be shown that $w$ satisfies
\begin{equation} \label{auxil11}
\|w\|_{H^2(\Om)}\lesssim \|u_{\mathcal{A}}-\hat{u}_{\mathcal{A},h}\|_{L^2(\Om)}.
\end{equation}
Let  $\hat{w}_h \in V_h$  be defined in Lemma~\ref{u_u_E}. Testing the conjugated \eqref{eqauxi1} by $u_{\mathcal{A}}-\hat{u}_{\mathcal{A},h}$ we get
\begin{align*}
\|u_{\mathcal{A}}-\hat{u}_{\mathcal{A},h}\|_{L^2(\Om)}&=(u_{\mathcal{A}}-\hat{u}_{\mathcal{A},h},-\triangle w+w)\\
&=(\nabla(u_{\mathcal{A}}-\hat{u}_{\mathcal{A},h}),\nabla w)+(u_{\mathcal{A}}-\hat{u}_{\mathcal{A},h},w)\\
&=a(u_{\mathcal{A}}-\hat{u}_{\mathcal{A},h},w )=a(u_{\mathcal{A}}-\hat{u}_{\mathcal{A},h},w-\hat{w}_h)\\
&\lesssim \|u_{\mathcal{A}}-\hat{u}_{\mathcal{A},h}\|_{H^1(\Om)}\|w-\hat{w}_h\|_{H^1(\Om)}\\
&\lesssim \|u_{\mathcal{A}}-\hat{u}_{\mathcal{A},h}\|_{H^1(\Om)}\frac{h}{p}\|w\|_{H^2(\Om)},
\end{align*}
which together with \eqref{auxil11} and \eqref{u_A_H^1} gives \eqref{u_A-pi}. The proof is completed.
\end{proof}

Next we prove Lemma~\ref{error2}.
We consider to approximate $\ue$ and $u_{\mathcal{A}}$, respectively.
First, from Lemma~\ref{u_u_E}, there are two functions $\hat{u}_{\mathcal{E},h} \in V_h$ and $\check{u}_{\mathcal{E},h} \in  W_h$ such that
\begin{align}
&\|\ue-\hat{u}_{\mathcal{E},h}\|_{H^j(\Om)}\lesssim \frac{ h^{2-j}}{p^{2-j}}C_{f,g},\ \ \ j=0,1, \label{u w H^2}\\
&\|\ue-\check{u}_{\mathcal{E},h}\|_{H^j(\mathcal{T}_h)}\lesssim \frac{ h^{2-j}}{p^{2-j}}C_{f,g},\ \ j=0,1,2.
\end{align}
On the other hand, from Lemma~\ref{u_u_A1}, there exist $\hat{u}_{\mathcal{A},h}\in V_h$ and $\check{u}_{\mathcal{A},h}\in   W_h$ satisfying \eqref{u_A-pi} and \eqref{u_A-pi_p} respectively. Let $\hat{u}_h=\hat{u}_{\mathcal{E},h}+\hat{u}_{\mathcal{A},h}$, $\check{u}_h=\check{u}_{\mathcal{E},h}+\check{u}_{\mathcal{A},h}$.
 It follows from the triangle inequality that
\begin{align}
\|u-\hat{u}_h\|_{H^j(\Om)}&\lesssim \Big(\frac{h^{2-j}}{p^{2-j}}+\frac{h^{1-j}}{p^{2-j}}\Big(\frac{kh}{\sigma p}\Big)^p\Big)C_{f,g},\ \ j=0,1, \label{u-u_h11} \\
\|u-\check{u}_h\|_{H^j(\mathcal{T}_h)}&\lesssim \Big(\frac{h^{2-j}}{p^{2-j}}+\frac{h^{1-j}}{p}\Big(\frac{kh}{\sigma p}\Big)^p\Big)C_{f,g},\ \ j=0,1,2. \label{u-u_h12}
\end{align}
Therefore, \eqref{u_err0} holds.

 \eqref{u_err2} can be proved by  following the same procedure as that for deriving \eqref{elapprox1b}. We omit the details. This completes the proof of Lemma~\ref{error2}.

\textbf{Acknowledgments.} The author would like to thank Professor J.M. Melenk for helpful communications.

\end{document}